\documentclass[12pt]{amsart}

\usepackage[
    left=0.8in,
    right=0.8in,
    top=0.8in,
    bottom=0.8in
]{geometry}
\usepackage{amssymb,amsthm, times, enumerate}
\usepackage{delarray,verbatim}
\usepackage{mathrsfs}

\usepackage{algorithm}

\usepackage{ifpdf}
\ifpdf
\usepackage[pdftex]{graphicx}
 \else
\usepackage[dvips]{graphicx}
 \fi

\usepackage{ifpdf}
\usepackage{color}
\definecolor{webgreen}{rgb}{0,.5,0}
\definecolor{webbrown}{rgb}{.8,0,0}
\definecolor{emphcolor}{rgb}{0.95,0.95,0.95}

\usepackage{hyperref}
\hypersetup{%
	colorlinks=true,
	linkcolor=webbrown,
	filecolor=webbrown,
	citecolor=webgreen,
	breaklinks=true}
\ifpdf \hypersetup{pdftex,
	pdfstartview=FitH, 
	bookmarksopen=true,
	bookmarksnumbered=true
} \else \hypersetup{dvips} \fi

\usepackage{mathrsfs}
\usepackage{booktabs}

\numberwithin{equation}{section}

\newtheorem{thm}{Theorem}[section]
\newtheorem{proposition}{Proposition}[section]
\newtheorem{corollary}{Corollary}[section]
\newtheorem{remark}{Remark}[section]
\newtheorem{lem}{Lemma}[section]
\newtheorem{exmp}{Example}[section]
\newtheorem{assum}{Assumption}[section]

\numberwithin{remark}{section} \numberwithin{proposition}{section}
\numberwithin{corollary}{section}

\newcommand{\diff}{{\rm d}}

\usepackage{mathtools}
\mathtoolsset{showonlyrefs=true}

\title{Continuous-time multi-armed bandits under random intervention times} 

\begin{document}

\author[K. Noba]{Kei Noba$^\dagger$}
\thanks{$\dagger$\, Department of Mathematics, Graduate School of Science, The University of Osaka. 1-1 Machikaneyama, Toyonaka
560-0043, Osaka,  Japan. Email: knoba@math.osaka-u.ac.jp}

	\author[J. L. P\'erez]{Jos\'e-Luis P\'erez$^{*}$}

		\thanks{$*$\, Department of Probability and Statistics, Centro de Investigaci\'on en Matem\'aticas A.C. Calle Jalisco
		s/n. C.P. 36240, Guanajuato, Mexico. Email: jluis.garmendia@cimat.mx}

	\author[K. Yamazaki]{Kazutoshi Yamazaki$^\ddagger$}
\thanks{$\ddagger$\, School of 
Mathematics and Physics, The University of Queensland, St Lucia,
Brisbane, QLD 4072, Australia. Email: k.yamazaki@uq.edu.au, qingyuan.zhang@uq.edu.au}

\author[Q. Zhang]{Qingyuan Zhang$^\ddagger$}

\begin{abstract}
This paper examines multi-armed bandits with $J$ independent arms in which actions are taken at random discrete times. When an arm is operated, it must remain active for a renewal inter-arrival time. For arms evolving as a L\'evy process, we provide an explicit characterization of the Gittins index, known to yield an optimal strategy. Furthermore, when the inter-arrival times are exponential and the arms evolve as a spectrally negative L\'evy process, a reflected spectrally negative L\'evy process, or a diffusion process, the Gittins index is explicitly characterized in terms of the scale function or diffusion characteristics, respectively. Convergence analysis and numerical experiments are performed to support the theoretical results.

		\noindent \small{\noindent  AMS 2020 Subject Classifications: 93E20, 60G51, 90B36 }
\end{abstract}

\maketitle

\section{Introduction}
The multi-armed bandit, introduced as an example of sequential allocation problems over a set of available actions, is a classic model for allocating scarce resources among competing projects under uncertainty. In this model, each arm represents a project that yields a reward when selected, and an agent allocates the resource, typically time, by choosing which project to operate, balancing immediate against future rewards. The problem was first formulated in discrete time, where arms are selected one at a time and rewards are discounted and summed; see, e.g., \cite{Gittins79,Mandelbaum98,Man1986,Whittle1980}. A continuous-time framework was subsequently developed; see \cite{Dalang90,Karatzas94,KM95,Mandelbaum98,Mandelbaum88}.

The multi-armed bandit has attracted considerable attention and led to a deep literature. An essential solution was obtained by Gittins and colleagues in discrete time \cite{Gittins79,Jones74}: they associate with each arm an index that depends only on that arm, and show that always selecting the arm with the largest current index is optimal. This index strategy, now known as the Gittins index strategy, reduces the original multi-dimensional problem to several one-dimensional problems, each characterizing the index of a single arm as the value of an optimal stopping problem driven by that arm \cite{Gittins79}. Solving these problems in closed form is rarely possible in discrete time, so the index must be evaluated numerically; in continuous time, where the index and its optimality were subsequently established, semi-explicit expressions are available under suitable assumptions on the arm dynamics \cite{KM95,Mandelbaum88}. See also \cite{CohenSolan14}, where the arm's type is unknown and the optimal strategy is a cut-off rule on the posterior belief.

In this work, we consider a version of the multi-armed bandit problem that lies between the discrete- and continuous-time settings. There are $J$ independent arms, each a continuous-time stochastic process; once selected, an arm remains active until the next renewal time, whose law may depend on the chosen arm, while unselected arms are frozen. In practice, the random-intervention times admit two interpretations. First, the random clock may be a feature of the project, as committing resources to an arm locks it in for a duration that is not known in advance, such as the length of a research phase or of an operating cycle; see Example \ref{Example: tech}. Second, the clock could be a feature of the controller, which can act only when a review or communication opportunity arises; see Examples \ref{Example: wireless} and \ref{Example: insurance}. In either case, reallocation is restricted to renewal times that are independent of the arm evolutions, though their law may depend on the selected arm. This problem was briefly considered in \cite{Va} as an extension of the discounted multi-armed bandit, establishing optimality of the Gittins index strategy without computing the index itself. A special case with common exponential renewal times and spectrally one-sided L\'evy arms was studied in \cite{PYM}, where the index was expressed explicitly, but its optimality was not addressed. Outside the spectrally one-sided exponential setting, no explicit characterization of the index appears to be available, and the effect of random-time intervention has not been investigated.

In view of the existing literature, this work develops the theory for the bandit problem under the random-intervention framework, and obtains tractable representations of the Gittins index across several classes of arm dynamics. We make three contributions. 

\begin{itemize}
    \item [(i)] L\'evy dynamics have been widely used to model controlled stochastic systems, including in recent control literature (e.g., \cite{KDo19,KDo19_drill,hernandez_et_al,perez_yamazaki_bensoussan,Zhang25}). For arms driven by general L\'evy processes and general renewal times, we characterize the Gittins index as the transform of the reward function against a probability measure, and establish the Fourier transform of this measure through a Wiener-Hopf factorization. Our analysis relies on the fluctuation theory of L\'evy processes (see \cite{Be,K} for continuous-time fluctuation identities and \cite{Albrecher,landriault_potential_2018,Lkabous,Yang} for identities under Poisson observation) and greatly generalizes \cite{PYM}.
    \item[(ii)] For arms driven by general L\'evy processes and exponential renewal times, we relate the random- and continuous-intervention Gittins index through an exact identity. Under mild conditions on the reward function and the L\'evy measure, this identity yields monotone pointwise convergence of the former to the latter as the intervention rate tends to infinity, with an explicit convergence rate governed by the small-time behavior of the L\'evy arms.
    \item [(iii)] The characterization of the Gittins index extends beyond L\'evy arms: under exponential renewal times we obtain a resolvent representation of the index requiring only monotonicity of the reward and the order-preserving property of the arm. This is applied to compute the index explicitly for reflected spectrally negative L\'evy processes (via scale functions) and one-dimensional diffusion processes (via the speed measure and Green function). Numerical experiments involving L\'evy and diffusion arms complement the analysis.
\end{itemize}

The remainder of the paper is organized as follows. Section \ref{Sect: prelim} formulates the multi-armed bandit problem. Section \ref{section_levy} characterizes the Gittins index for general L\'evy processes. Section \ref{sec_exponential} treats the exponential-renewal case, characterizing the index for L\'evy processes, reflected L\'evy processes, and diffusion processes, together with convergence analysis. Section \ref{Sect: experiment} presents numerical analysis. Section \ref{Sect: conclusion} concludes with directions for future research. Lengthy proofs are deferred to the appendix.

\section{Preliminaries}\label{Sect: prelim}

\subsection{Model}

Our model consists of $J \in \mathbb{N}$ independent arms that can be selected in any order and one at a time. Let $\mathcal{J} := \{1, \ldots, J\}$. When an arm is selected, a reward is collected, and the arm must be operated without interruption for an independent time interval, the length of which follows an (arm-dependent) distribution $(G^j; j \in \mathcal{J})$. In the sequel, a \textit{period} $t \in \mathbb{N}_0 \coloneqq \mathbb{N} \cup \{0\}$ begins when an arm is selected and ends when it is released.

Each arm $j \in \mathcal{J}$ is associated with a sequence
\[
\{ X^j_s, W^j_s, \mathcal{F}^j_{s}; s \in \mathbb{N}_0 \},
\]
defined on a complete probability space $(\Omega, \mathcal{F}, \mathbb{P})$. Here, $X^j_s$ for $s \in \mathbb{N}_0$ is the reward from arm $j$ when it is selected after $s$ previous operations, and $W^j_s$ models the duration for which it must be active. We assume $W^j = (W^j_s; s \in \mathbb{N}_0)$ is an i.i.d.\ sequence of strictly positive random variables with a common distribution $G^j$. The collection of sigma-algebras $\mathbb{F}^j := (\mathcal{F}^j_s; s \geq 0)$ is a $\mathbb{P}$-completed filtration. For $s \in \mathbb{N}_0$, we assume $X^j_s$ is $\mathcal{F}^j_{s}$-measurable and $W^j_{s}$ is $\mathcal{F}^j_{s}$-independent, but $W^j_{s}$ is $\mathcal{F}^j_{s+1}$-measurable. Since the arms are independent, we have that $\mathcal{F}^j := \bigvee_{s \geq 0} \mathcal{F}^j_s$ are independent for $j \in \mathcal{J}$.

Our model of a  bandit strategy follows that of \cite{Man1986}. At the beginning of each period $t \in \mathbb{N}_0$, exactly one arm  is selected. A strategy $\pi$ is given by an $\mathbb{N}_0^J$-valued process $(\pi(t) = [\pi^1(t), \ldots, \pi^J(t)]; t \in \mathbb{N}_0)$, where $\pi^j(t)$ denotes the number of times arm $j$ has been selected during periods $0, 1, \ldots, t-1$ (excluding $t$). Thus, $\pi$ must satisfy:
\begin{align}
    &\pi(0) = [0, \ldots, 0], \label{Eq: pi0}\\
    &t = \pi^1(t) + \cdots + \pi^J(t), \quad t \in \mathbb{N}_0, \notag\\
    &\pi(t+1) = \pi(t) + \mathbf{e}_j \quad \textrm{for some } j \in \mathcal{J}, \label{successor}
\end{align}
where $\mathbf{e}_j$ is the $J$-dim vector with the $j$th element being $1$ and the others $0$. In addition, for $\mathbf{s} = (s^1, \ldots, s^J) \in \mathbb{N}_0^J$,
\begin{equation}
    \{\pi(t) = \mathbf{s} \} \in \mathcal{F}_{\mathbf{s}} \coloneqq \mathcal{F}_{s^1}^1 \vee \cdots \vee \mathcal{F}_{s^J}^J.\label{Eq: measurable pi}
\end{equation}
With this formulation of strategy $\pi$, the information available at the onset of period $t \in \mathbb{N}_0$ is $\mathcal{F}_{\pi(t)}$, which is well-defined, as confirmed in \cite[Section 3]{Man1986}. In particular, for $t \geq 1$, $\pi(t)$ is $\mathcal{F}_{\pi(t-1)}$-measurable, which follows from \cite[Theorem 3]{Man1986}; together with the independence of $(\mathcal{F}^j)_{j \in \mathcal{J}}$ and the $\mathbb{P}$-completeness of $\mathcal{F}_{\mathbf{s}}$ for $\mathbf{s} \in \mathbb{N}_0^J$, this is equivalent to (given \eqref{successor}), $\{\pi(t+1) - \pi(t) = \mathbf{e}_j, \pi(t) = \mathbf{s} \} \in \mathcal{F}_{\mathbf{s}}$, for all  $j \in \mathcal{J}, \mathbf{s} \in \mathbb{N}_0^J, t \in \mathbb{N}_0$. Denote by $\mathrm{j}_t\in \mathcal{J}$ the arm selected at the onset of period $t$. This random variable can be recovered from $(\pi(s); s \leq t)$ and is therefore $\mathcal{F}_{\pi(t)}$-measurable. By \eqref{successor},  $\pi(t)$ is determined by $\pi(t - 1)$ and $\mathrm{j}_t$, for $t \geq 1$. We use $\Pi$ to denote the set of admissible strategies, including every $\pi$ that satisfies \eqref{Eq: pi0}--\eqref{successor} and the measurability condition in \eqref{Eq: measurable pi}.

For $t \in \mathbb{N}_0$, the \textit{real process time} of period $t$ is $[\mathcal{T}_{t}, \mathcal{T}_{t + 1})$ with $\mathcal{T}_0 = 0$ and
\[
\mathcal{T}_t = \mathcal{W}_0 + \mathcal{W}_1 + \cdots + \mathcal{W}_{t - 1}, \quad t = 1,2, \ldots
\]
Here, $\mathcal{W}_t := W^{\mathrm{j}_t}_{\pi^{\mathrm{j}_t}(t)}$ is the duration of period $t$, during which the selected arm $\mathrm{j}_t$ is active. For $t \in \mathbb{N}_0$, we have that $\mathcal{T}_{t}$ is $\mathcal{F}_{\pi(t)}$-measurable. In particular, $\mathcal{T}_0 = 0 \in \mathcal{F}_{\pi(0)} = \bigvee_{j \in \mathcal{J}} \mathcal{F}_0^j$ and $\mathcal{T}_1 = W^{\mathrm{j}_{0}}_{\pi^{\mathrm{j}_{0}}(0)} \in \mathcal{F}_{\pi(1)}$.

Let $\mathbb{P}^\pi$ be the law induced by an admissible strategy $\pi$ given $\mathcal{F}_{\pi(0)}$ and $\mathbb{E}^\pi$ its expectation. Fix a discount factor $q > 0$. The expected discounted reward is 
\begin{equation}
    V(\pi) = \mathbb{E}^\pi \Big[ \sum_{t = 0}^\infty \mathcal{X}_t \int_{\mathcal{T}_{t}}^{\mathcal{T}_{t + 1}}  e^{-qr} {\rm d} r 
    \Big], \label{V_pi}
\end{equation}
where $\mathcal{X}_{t} :=  X^{\mathrm{j}_t}_{\pi^{\mathrm{j}_t}(t)} \in \mathcal{F}_{\pi(t)}$ is the reward from the arm selected in period $t$, active on $[\mathcal{T}_t, \mathcal{T}_{t+1})$. With the generality of $\mathcal{X}_{t}$, various models are incorporated; see Remark \ref{remark_problem2}. We want to compute the maximal value $\sup_{\pi \in \Pi} V(\pi)$ and an optimal strategy $\pi^* \in \Pi$ that attains it. For $j \in \mathcal{J}$, let $T^j_0 = 0$ and $T^j_s := W^j_0 + \cdots +W^j_{s - 1}$ for $s \geq 1$.
\begin{assum}\label{asm: integrability} We assume
    $\mathbb{E} \big[ \sum_{s=0}^{\infty} |X^j_s |\int_{T_s^j}^{T_{s+1}^j}  e^{-qr} {\rm d} r \big]$ is finite for $j \in \mathcal{J}$  to ensure that \eqref{V_pi} is finite.
\end{assum}
\begin{remark}\label{remark1} 
If $G^j \equiv G$ for all $j \in \mathcal{J}$, then $\mathcal{T}_{t + 1} - \mathcal{T}_{t} = \mathcal{W}_t$ is independent of $\mathcal{F}_{\pi(t)}$, and
\begin{align*}
    V(\pi) &= \mathbb{E}^\pi \Bigg[ \sum_{t = 0}^\infty \mathcal{X}_t \mathbb{E}^\pi \Bigg[ \int_{\mathcal{T}_{t}}^{\mathcal{T}_{t + 1}}  e^{-qr} {\rm d} r \Bigg| \mathcal{F}_{\pi(t)}\Bigg] \Bigg] = \mathbb{E} \Bigg[ \int_{0}^{W}  e^{-qr} {\rm d} r \Bigg] \mathbb{E}^\pi \Bigg[ \sum_{t = 0}^\infty e^{-q \mathcal{T}_{t}}\mathcal{X}_t\Bigg],
\end{align*}
for $W \sim G$. Hence, the problem reduces to maximizing $\mathbb{E}^\pi[ \sum_{t = 0}^\infty e^{-q \mathcal{T}_t} \mathcal{X}_t]$, the problem considered in \cite{PYM}.
\end{remark}

\subsection{Gittins index}
Fix arm $j \in \mathcal{J}$. Define the Gittins index process
\begin{align}\label{Gittins_general}
\begin{split}
    \Gamma^j_s &:= \operatorname{ess~sup}_{\tau > s} \frac {\mathbb{E} \left[ \sum_{u = s}^{\tau - 1} e^{-q (W^j_s + \cdots + W^j_{u - 1})} X^j_u \Big| \mathcal{F}^j_s \right]} {\mathbb{E} \left[ \sum_{u = s}^{\tau - 1} e^{-q (W^j_s + \cdots + W^j_{u - 1})}  \Big| \mathcal{F}^j_s \right]} =
    \operatorname{ess~sup}_{\tau > s} \frac {\mathbb{E} \left[\sum_{u = s}^{\tau - 1} e^{-q T^j_{u}} X^j_u\Big| \mathcal{F}^j_s \right]} {\mathbb{E} \left[ \sum_{u = s}^{\tau - 1} e^{-q T^j_{u}}  \Big| \mathcal{F}^j_s \right]},
\end{split}
\end{align}
where the maximization is taken over all $\mathbb{F}^j$-stopping times larger than $s$. The existence of an optimal stopping time is not guaranteed in general. Proposition \ref{prop_opt_stopping_time} gives conditions under which it can be identified. Our setting satisfies the conditions required to apply \cite[Theorem 3.1]{Va}, which establishes the optimality of the Gittins index policy.

\begin{thm}
\label{Thm: sec2 Gittin index general}
The Gittins index strategy, which at each onset of period $t \in \mathbb{N}_0$ operates arm
\[
\arg \max_{j \in \mathcal{J}} \Gamma^j_{\pi^{j}(t)},
\]
is optimal.
\end{thm}

\begin{proof}
For $s \in \mathbb{N}_0$ and arm $j \in \mathcal{J}$, define
\begin{align*}
    &\tilde{\Gamma}^j_s \coloneqq \operatorname{ess~sup}_{\tau > s}
     \frac {\mathbb{E} \left[ \sum_{u = s}^{\tau-1} e^{-q (W^j_s + \cdots + W^j_{u-1})} X^j_u \int_0^{W^j_u}  e^{-qr} {\rm d} r \Big| \mathcal{F}^j_s \right]} {\mathbb{E} \left[ \sum_{u = s}^{\tau-1} e^{-q (W^j_s + \cdots + W^j_{u-1})} \int_0^{W^j_u}  e^{-qr} {\rm d} r \Big| \mathcal{F}^j_s \right]}.
\end{align*}
By \cite[Theorem 3.1]{Va}, a strategy that operates an arm with the highest index $\tilde{\Gamma}^j_s$ is optimal. Since $\tau$ is a stopping time,
\begin{align*}
    &\mathbb{E} \left[ \sum_{u = s}^{\tau-1} e^{-q (W^j_s + \cdots + W^j_{u-1})} X^j_u \int_0^{W^j_u}  e^{-qr} {\rm d} r \Big| \mathcal{F}^j_s \right] \\
    &=\mathbb{E} \left[ \sum_{u = s}^{\infty} 1_{\{u \leq \tau-1\}} e^{-q (W^j_s + \cdots + W^j_{u-1})} X^j_u \right. \left.\mathbb{E} \left[ \int_0^{W^j_u}  e^{-qr} {\rm d} r \Big| \mathcal{F}^j_u \right] \Big| \mathcal{F}^j_s \right] \\
    &= \mathbb{E} \left[ \int_0^{W^j_1}  e^{-qr} {\rm d} r  \right] \mathbb{E} \left[ \sum_{u = s}^{\infty} 1_{\{u \leq \tau-1\}} e^{-q (W^j_s + \cdots + W^j_{u-1})} X^j_u  \Big| \mathcal{F}^j_s \right],
\end{align*}
where, in the last step, we used the assumption that $(W^j_u)$ are i.i.d. random variables and $W^j_u$ is independent of $\mathcal{F}^j_u$ for each $u \in \mathbb{N}_0$. Applying the same treatment to the denominator, we obtain $\tilde{\Gamma}^j_s = \Gamma^j_s$. The second equality of \eqref{Gittins_general} follows from the definition $T^j_s$ and the measurability of $(W^j_u)_{u < s}$ with respect to $\mathcal{F}^j_s$.
\end{proof}

\subsection{Markovian case} \label{subsection_Markov}
For $j \in \mathcal{J}$, let $Y^j := (Y^j(t); t \geq 0)$ be an $\mathbb{R}$-valued continuous-time Markov process and $R^j: \mathbb{R} \to \mathbb{R}$ be a measurable function. Denote by $X^j \coloneqq (R^j(Y^{j}(T^j_{s})); s \in \mathbb{N}_0)$ the corresponding reward process, and let $\mathcal{F}^j_s$ be the $\mathbb{P}$-completion of $\sigma(Y^{j}(T_{0}^j), \ldots, Y^j(T^j_{s}), W^j_{0}, \ldots, W^j_{s-1})$, for $s \in \mathbb{N}_0$. In this case, \eqref{V_pi} reduces to
\begin{align} \label{V_pi_Markov}
    V(\pi) = \mathbb{E}^\pi \Big[ \sum_{t = 0}^\infty R^{\mathrm{j}_t}(Y^{\mathrm{j}_t}(T^{\mathrm{j}_t}_{\pi^{\mathrm{j}_t}(t)})) \int_{\mathcal{T}_{t}}^{\mathcal{T}_{t + 1}}  e^{-qr} {\rm d} r\Big].
\end{align}
Let $\mathbb{P}_x^j$ be the probability measure for arm $j$ under which $Y^j(0) = x$. For $x = 0$, we simply write $\mathbb{P}^j$. Under this Markovian setting, the Gittins index \eqref{Gittins_general} reduces to
\begin{align*}
    \Gamma^j_s = \Gamma^j(Y^j(T^j_{s})), \quad s \in \mathbb{N}_0,
\end{align*}
where, for $x \in \mathbb{R}$, 
\begin{align} \label{Gamma_x_Markov}
    \Gamma^j(x) := \sup_{\tau > 0} \frac {\mathbb{E}_x^j \left[ \sum_{s = 0}^{\tau - 1} e^{-q T^j_{s}} R^j(Y^j(T^j_{s})) \right]} {\mathbb{E}_x^j \left[ \sum_{s = 0}^{\tau - 1} e^{-q T^j_{s}} \right]}. 
\end{align}
The above optimal stopping problem can be explicitly solved under the monotonicity and comparison conditions imposed in Assumption \ref{assump_comparison}. These assumptions are standard in continuous-time frameworks such as \cite{Mandelbaum98}.
\begin{assum} \label{assump_comparison}
For $j \in \mathcal{J}$, $R^j: \mathbb{R} \to \mathbb{R}$ is strictly increasing  and $Y^j$ is order-preserving in the sense that, given two paths $Y^j_x$ and $Y^j_y$, if $Y^j_x(0) = x < y = Y^j_y(0)$, then $Y^j_x(t) \leq Y^j_y(t)$ for all $t \geq 0$ almost surely.  
\end{assum}
Increasing reward functions arise naturally, e.g., when the state represents a project's quality or performance, or when the reward is framed as the negative of a cost to be minimized. Examples \ref{Example: wireless}, \ref{Example: insurance}, and \ref{Example: tech} develop these settings; relaxing Assumption \ref{assump_comparison} is discussed in Section \ref{Sect: conclusion}.
\begin{proposition}[{Generalization of \cite[Section 3.3]{PYM}}]\label{prop_opt_stopping_time}
    Under Assumption \ref{assump_comparison}, the optimal stopping problem \eqref{Gamma_x_Markov} is solved by
    \[
    \tau_x^- = \tau_x^{-,j} := \inf \{ s \geq 1:   Y^j(T^j_{s}) \leq x \}.
    \] 
Hence,
    \begin{equation}
    \Gamma^j(x) = \frac{\mathbb{E}_x [\sum^{\tau_x^- - 1}_{k = 0} e^{-qT^j_k} R^j(Y^j(T^j_k))]}{\mathbb{E}_x [\sum^{\tau_x^- - 1}_{k = 0} e^{-q T^j_k}]}. \label{Gamma_tilde}
\end{equation}
\end{proposition}
\begin{proof} 
We provide only a sketch of the proof, as most of the arguments follow those in \cite{PYM}. Let $\hat{\Gamma}^j(x)$ denote the right hand side of \eqref{Gamma_tilde}. Using Snell's optimal stopping theory and following the same step as in the proof of \cite[Lemma 1]{PYM}, the problem \eqref{Gamma_x_Markov} is solved by the stopping time
\begin{equation*}
    \hat{\tau}^-_x=\inf\{s\geq1: \hat{\Gamma}^j(Y^j(T^j_s)) \leq \hat{\Gamma}^j(x)\}.
\end{equation*}
Under Assumption \ref{assump_comparison}, $x \mapsto\hat{\Gamma}^j(x)$ is strictly increasing, and thus $\hat{\tau}^-_x = \tau^-_x$, consequently $\hat{\Gamma}^j = \Gamma^j$, as desired.
\end{proof}
Proposition \ref{prop_opt_stopping_time} shows how the reward's monotonicity and the arm's order-preserving property identify the optimal stopping time in \eqref{Gamma_x_Markov}: since the objective is the average discounted reward per unit of discounted time, continuing at $y$ is desirable only if $\Gamma^j(y) > \Gamma^j(x)$, and the resulting monotonicity of $\Gamma^j$ means the problem continues while the arm remains strictly above its starting level, stopping at the first observation at or below it. In the remainder of the paper we work under Assumption \ref{assump_comparison} and characterize the Gittins index \eqref{Gamma_x_Markov} for several classes of continuous-time Markov processes.

\begin{remark} \label{remark_tau_tilde}
In Proposition \ref{prop_opt_stopping_time}, it can be verified that the maximum \eqref{Gamma_x_Markov} is also achieved by $\tilde{\tau}_x^- := \inf \{ s \geq 1: Y^j(T^j_{s}) < x \}$ with strict inequality.
\end{remark}
\begin{remark} \label{remark_problem2}
	As a variant of \eqref{V_pi_Markov}, one may maximize
	\begin{equation*}
	\widetilde{V}(\pi) := \mathbb{E}^\pi \left[ \sum_{t = 0}^\infty  \int_{\mathcal{T}_{t}}^{\mathcal{T}_{t + 1}}  e^{-qu} r^{\mathrm{j}_t}(Y^{\mathrm{j}_t}(T^{\mathrm{j}_t}_{\pi^{\mathrm{j}_t}(t)} + u - \mathcal{T}_t)) {\rm d} u \right],
	\end{equation*}
for strictly increasing $r^{j}: \mathbb{R} \to \mathbb{R}$, $j \in \mathcal{J}$, so that the reward evolves continuously with the state of the selected arm. The strong Markov property gives
	\begin{align*}
		\widetilde{V}(\pi) &= \mathbb{E}^\pi \Biggl[ \sum_{t = 0}^\infty  \mathbb{E}^\pi \Biggl[ \int_{\mathcal{T}_{t}}^{\mathcal{T}_{t + 1}}  e^{-qu}  r^{\mathrm{j}_t}(Y^{\mathrm{j}_t}(T^{\mathrm{j}_t}_{\pi^{\mathrm{j}_t}(t)} + u - \mathcal{T}_t)) {\rm d} u \Bigg| \mathcal{F}_{\pi(t)} \Biggr] \Biggr] \\
		&= \mathbb{E}^\pi \left[ \sum_{t = 0}^\infty  e^{-q \mathcal{T}_{t}} R^{\mathrm{j}_t} (Y^{\mathrm{j}_t}(T^{\mathrm{j}_t}_{\pi^{\mathrm{j}_t}(t)}))\right],
	\end{align*}
	for 
	$
	R^j(x) := \mathbb{E}_x^j [ \int_{0}^{W_1^j}  e^{-qt} r^j(Y^j(t)) {\rm d} t ]$,  $j \in \mathcal{J}$ and $x \in \mathbb{R}$.
	In view of Remark \ref{remark1}, for arm-independent $(G^j)_{j}$, this is a special case of \eqref{V_pi_Markov}. Moreover, if $W^j_1 \sim \exp(\lambda)$, $R^{j}(x) = \mathbb{E}_{x}^j \left[ \int_{0}^{\infty} e^{-(q+\lambda)t} r^{j}(Y^{j}(t)) {\rm d} t \right]$. This formulation arises naturally when reward accrues continuously with the arm's state, as in Example \ref{Example: insurance}.
\end{remark}

\section{L\'evy process case} \label{section_levy}
In this section, we take each $Y^j$ to be a one-dimensional L\'evy process. The following example motivates this modeling choice.
\begin{exmp}[Wireless-powered sensor scheduling]\label{Example: wireless}
    A base station wirelessly charges $J$ energy-harvesting sensors and, at the end of each session, decides which sensor to serve next; its hardware lets it drive only one dedicated wireless power-transfer link at a time. While connected, sensor $j$'s battery level $Y^j$ rises according to a (possibly jump-driven, due to fluctuating channel gain) L\'evy process; once disconnected, the sensor duty-cycles into a sleep mode with negligible standby draw, so it draws no power and $Y^j$ is frozen until reconnected. The session's control signaling rides on a shared, contention-based band, so its length is random -- prolonged by collisions and retransmissions rather than by how much charge sensor $j$ has accumulated -- with a law $G^j$ that may depend on the sensor's protocol class. Since a higher battery level lowers the sensor's outage probability, the reward is naturally increasing in $Y^j$, e.g., $R^j(x) = c^j x$ with $c^j > 0$. This is a random-intervention instance of control under communication constraints \cite{HespanhaSurvey}, and is closely related to networked control formulations in which sensor/controller links open and close according to independent renewal processes \cite{AntunesHespanhaSilvestre}.
\end{exmp}
Let $\mu$ be the probability measure on $([0,\infty), \mathcal{B}[0,\infty))$ such that, for a bounded measurable function $f:\mathbb{R} \to \mathbb{R}$,
\begin{equation}\label{mu_1}
    \int_{[0,\infty)}f(y)\mu({\rm d} y) = \frac{\mathbb{E} [ \sum_{k=0}^{\tau_0^- -1 }e^{-q T_k} f(Y(T_k))]}{\mathbb{E} [ \sum_{k=0}^{\tau_0^ --1} e^{-q T_k} ]}.
\end{equation}
Here and throughout the remainder of this section, we omit the superscript $j$. By Proposition \ref{prop_opt_stopping_time} and the spatial homogeneity of L\'evy processes, the Gittins index \eqref{Gamma_x_Markov} becomes $\Gamma(x) = \int_{[0, \infty)} R(x + y) \mu({\rm d} y)$ for $x \in \mathbb{R}$.
It therefore suffices to characterize $\mu$, which we do through its Fourier transform using tools from the fluctuation theory of L\'evy processes. For a probability measure $m$ on $([0,\infty), \mathcal{B}[0, \infty))$ and $\theta \in \mathbb{R}$, we write $\widehat{m}(\theta) := \int_{[0,\infty)} e^{\mathrm{i} \theta y} m({\rm d} y)$.

Define $S_n := Y(T_n)$ with $S_0=0$ under $\mathbb{P}$ and $S_n=\sum_{k=1}^n(Y(T_k)-Y(T_{k-1}))$ for $n \geq 1$. Let $N :=\inf\{n\geq 1: S_n> 0\}$ and $\hat{N}:=\inf\{n\geq 1: S_n \leq 0\}$. Note that $\tau_0^- = \hat{N}$ under $\mathbb{P}$. For $q > 0$, $\theta\in\mathbb{R}$, and $r \in (0,1]$, let
\begin{equation} \label{kappa_varphi}
    \kappa(q,\theta) := \mathbb{E}[e^{-qT_1+\mathrm{i} \theta S_1}], \quad \varphi(q,\theta; r) := \mathbb{E}[r^N e^{-q T_{N}+\mathrm{i} \theta S_N}], \quad \hat{\varphi}(q,\theta;r) := \mathbb{E}[r^{\hat{N}} e^{-q T_{\hat{N}}+\mathrm{i} \theta S_{\hat{N}}}].
\end{equation} 
Also, define $\varphi(q,\theta) := \varphi(q,\theta; 1)$ and $\hat{\varphi}(q,\theta) := \hat{\varphi}(q,\theta; 1)$.

Fix $r \in (0,1)$ and $p:= 1-r$.
Consider an $\mathbb{N}_0$-valued geometric random variable $\Xi_{p}$ with parameter $p$ (so $\mathbb{P}(\Xi_p > n) = r^{n+1}$), independent of the process $(T_n,S_n)_{n\geq1}$. 
Define the first visit of the random walk $S$ to its maximum until time $\Xi_{p}$ and the last visit of $S$ to its minimum by
\begin{equation*}
\begin{aligned}
    G &:= \min\left\{k=0,1\dots,\Xi_{p}:S_k=\max_{i=0,1,\dots,\Xi_{p}}S_i\right\},\\
    D &:= \max\left\{k=0,1\dots,\Xi_{p}:S_k=\min_{i=0,1,\dots,\Xi_{p}}S_i \right\}.
\end{aligned}
\end{equation*}
We provide a version of the Wiener-Hopf factorization for our setting, which can be obtained by adapting the proof of Theorem 1 in \cite{greenwood} (see also \cite{kyprianou_note}). The full proof is given in Appendix \ref{proof_wf}. 
\begin{lem} \label{WF}
The pair $(T_G, S_G)$ is independent of $(T_{\Xi_p} - T_G, S_{\Xi_p} - S_G)$, which has the same distribution as $(T_D, S_D)$. Additionally, for $r \in (0,1]$,
	\begin{align}\label{spitzer_ran_1}
		1 - r \kappa(q, \theta) = (1 - \varphi(q, \theta; r)) (1-\hat{\varphi}(q, \theta;r)).
	\end{align}
\end{lem}
The Fourier transform of the measure $\mu$ given in \eqref{mu_1} is established in the following proposition.
\begin{proposition}\label{Gitt_gral}
For $\theta \in \mathbb{R}$, 
\begin{equation*}
    \widehat{\mu}(\theta)= \frac{1-\hat{\varphi}(q,\theta)}{1-\kappa(q,\theta)} \frac {1-\kappa(q,0)} {1-\hat{\varphi}(q,0)}  =\frac{1-\varphi(q,0)}{1- \varphi(q,\theta)}.
\end{equation*}
\end{proposition}
\begin{proof}
Define the ascending and weak descending ladder height processes $(l,H):=\{(l_n,H_n) := (T_{\tau_n}, S_{\tau_n}); n \in \mathbb{N}_0\}$ and $(\hat{l},\hat{H}):=\{(\hat{l}_n,\hat{H}_n) := (T_{\hat{\tau}_n}, S_{\hat{\tau}_n}); n \in \mathbb{N}_0 \}$. Here, $\tau_n$ and $\hat{\tau}_n$ are recursively given by $\tau_0= \hat{\tau}_0= 0$, and for $n\geq1$,
	\begin{align*}
		\tau_n&:=
		\begin{cases}
			\inf \{k>\tau_{n-1}: S_k > H_{n-1}\}, &  \text{ if $\tau_{n-1}<\infty$}, \\
			\infty, & \text{ if $\tau_{n-1}=\infty$,} 
		\end{cases}
\\
		\hat{\tau}_n&:=
		\begin{cases}
			\inf \{k >\hat{\tau}_{n-1}: S_k  \leq \hat{H}_{n-1}\}, & \text{ if $\hat{\tau}_{n-1}<\infty$}, \\
			\infty, &  \text{ if $\hat{\tau}_{n-1}=\infty$}.
		\end{cases}
\end{align*}
Let $l_n = H_n = \infty$ when $\tau_n = \infty$ and $\hat{l}_n = \hat{H}_n = \infty$ when $\hat{\tau}_n = \infty$. By the strong Markov property, the process $(l,H)$ has independent and stationary increments, and so does  $(\hat{l},\hat{H})$, with 
	\begin{equation} \label{stationary_increments}
			(l_{n+1} - l_n, H_{n+1} - H_n) \sim (T_N, S_N), \quad (\hat{l}_{n+1} - \hat{l}_n, \hat{H}_{n+1} - \hat{H}_n) \sim (T_{\hat{N}}, S_{\hat{N}}).
	\end{equation}
   Fix $q > 0$ and $\theta \in \mathbb{R}$. We have 
\begin{align}\label{wh3}
\begin{split}
    &\mathbb{E}\left[\sum_{n=0}^{\infty}e^{-qT_n+\mathrm{i} \theta S_n }\right]= \mathbb{E}\left[\sum_{n=0}^{\infty} 1_{\{\hat{\tau}_n < \infty \}}\mathbb{E}\left[
    \sum_{m= \hat{\tau}_n}^{\hat{\tau}_{n+1}-1} e^{-qT_m+\mathrm{i} \theta S_m } \Big| \mathcal{F}_{\hat{\tau}_n} \right] \right],
\end{split}
\end{align}
where $\mathcal{F}_{\hat{\tau}_n} := \{ A \in \mathcal{F}: A \cap \{ \hat{\tau}_n = k \} \in \sigma(T_1, S_1 , \ldots, T_k, S_k), k  \in \mathbb{N}_0 \}$. Because $S_n = Y(T_n)$ for $n \ge 1$, and since $Y$ has independent and stationary increments and the family $\{T_n\}_{n \in \mathbb{N}}$ is independent of $Y$, it follows that on the set $\{\hat{\tau}_n < \infty\}$,
\begin{align}\label{wh4}
\begin{split}
    \mathbb{E}\left[\sum_{m= \hat{\tau}_n}^{\hat{\tau}_{n+1}-1} e^{-qT_m+\mathrm{i} \theta S_m} \Big| \mathcal{F}_{\hat{\tau}_n} \right] &=e^{-q \hat{l}_n+ \mathrm{i} \theta \hat{H}_n} \mathbb{E}\left[\sum_{m= \hat{\tau}_n}^{\hat{\tau}_{n+1}-1}e^{-q(T_m- \hat{l}_n)+\mathrm{i} \theta (S_m- \hat{H}_n))}\Big| \mathcal{F}_{\hat{\tau}_n}\right] \\
    & =e^{-q \hat{l}_n+\mathrm{i}  \theta \hat{H}_n} A(q, \theta)
\end{split}
\end{align}
where
\[
    A(q,\theta) := \mathbb{E}\left[\sum_{m=0}^{{\tau_0^--1}}e^{-qT_m+\mathrm{i} \theta S_m}\right].
\]
Using \eqref{wh4} in \eqref{wh3} gives
\begin{align}\label{wh5}
    \mathbb{E}\left[\sum_{n=0}^{\infty}e^{-qT_n+\mathrm{i} \theta S_n}\right]
    &=\mathbb{E}\left[\sum_{n=0}^{\infty}e^{-q \hat{l}_n+\mathrm{i} \theta \hat{H}_n}\right] A(q,\theta).
\end{align}
By the independence and stationarity of the increments
	$
	\bigl(T_{n+1}-T_n,\; S_{n+1}-S_n; n \in \mathbb{N}_0),
	$
\begin{align}\label{wh1}
\mathbb{E}\left[\sum_{n=0}^{\infty}e^{-qT_n+\mathrm{i} \theta S_n}\right]=\sum_{n=0}^{\infty}\mathbb{E}\left[e^{-qT_n+\mathrm{i} \theta S_n}\right] 
=\sum_{n=0}^{\infty}\kappa(q,\theta)^n =(1-\kappa(q,\theta))^{-1}.
\end{align}
Eq. \eqref{stationary_increments} together with \eqref{kappa_varphi} yields $\mathbb{E}[e^{-q \hat{l}_n+\mathrm{i} \theta \hat{H}_n}]=\mathbb{E}[e^{-q \hat{l}_1 +\mathrm{i} \theta \hat{H}_1}]^n=\hat{\varphi}(q,\theta)^n,
$ and so 
\begin{align*}
\mathbb{E}\left[\sum_{n=0}^{\infty}e^{-q \hat{l}_n+\mathrm{i} \theta \hat{H}_n}\right]=\sum_{n=0}^{\infty}\hat{\varphi}(q,\theta)^n= (1-\hat{\varphi}(q,\theta))^{-1}.
\end{align*}
In view of the last identity, \eqref{wh5}, \eqref{wh1}, and Lemma \ref{WF},
\begin{align*}
A(q,\theta)=\frac{1-\hat{\varphi}(q,\theta)}{1-\kappa(q,\theta)} = (1-\varphi(q,\theta))^{-1}.
\end{align*}
The claim holds by applying this to \eqref{mu_1}.
\end{proof}

\begin{exmp}Suppose that $Y$ is spectrally one-sided, so that the closed-form expression of $\varphi(q, \mathrm{i}\theta)$ is given by \cite[Theorem 3.1]{Albrecher}. Substituting it into Proposition \ref{Gitt_gral} and applying partial fraction decomposition to the resulting Fourier transform recovers the expression for $\mu^\lambda$ obtained in \cite[Section 5.2 ]{PYM}, which Proposition \ref{Gitt_gral} generalizes.
\end{exmp}

The following extends \cite[Lemma 1, XVIII.3]{Feller} and provides a way to compute the Fourier transform of Proposition \ref{Gitt_gral}; it is used in Proposition \ref{Lemma: GI decomp} to relate the index to its continuous-intervention counterpart. See Appendix \ref{use_lemma} for its proof.
\begin{proposition}
\label{imp_lem} 
	For $0 < r\leq 1$ and $q,\theta \geq 0$,
    \begin{equation*}
		- \log (1-\varphi(q,\theta; r)) = \sum_{n=1}^{\infty}\frac{r^n}{n}\int_{(0,\infty)\times[0,\infty)}e^{-qt+\mathrm{i} \theta x}\mathbb{P}\left(Y(T_n)\in {\rm d} x, T_n\in {\rm d} t\right).
	\end{equation*}
\end{proposition}
\section{Exponential inter-arrival case} \label{sec_exponential}
In this section, we continue in the Markovian setting of Section \ref{subsection_Markov} and impose the following assumption.
\begin{assum} \label{assump_exponential}
    $W_n^j \sim \exp(\lambda^j)$ for each $n \in \mathbb{N}_0$.
\end{assum}
\begin{remark}
    In view of Eq. \eqref{Gamma_tilde} and Assumption \ref{asm: integrability}, as $\lambda^j \downarrow 0$,  we have $T_1^j \uparrow \infty$ and thus $\Gamma^{j}(\cdot) \to R^j(\cdot)$.
\end{remark}
For the rest of this section, we omit the superscript $j$ for conciseness and to indicate the dependence on $\lambda$, we write $\Gamma$ as $\Gamma^\lambda$.

\subsection{L\'evy process case.} \label{subsec_levy_case} 
We first specialize the L\'evy case of Section \ref{section_levy} to Assumption \ref{assump_exponential}. To indicate the dependence on $\lambda$, we write $\mu$ in \eqref{mu_1} as $\mu^\lambda$. Proposition \ref{Gitt_gral} already characterizes the Gittins index through $\mu^\lambda$. We study the convergence of the index as $\lambda \uparrow \infty$, establishing its monotonicity, uniformity, and rate.

Denote by $(L(t))_{t\geq0}$ a local time at the maximum of the process $Y$ (\cite[Sections IV.2-4]{Be}), and define its inverse by
\begin{align*}
	L^{-1}(t):=
	\begin{cases}
		\inf\{s>0:L(s)>t\}, 
		& t<L(\infty),\\
		\infty, & \text{otherwise}.
	\end{cases}
\end{align*}
Set $H(t):=Y(L^{-1}(t))$ for $t<L(\infty)$, and $H(t):=\infty$ otherwise. Then, as shown in  \cite[Section 6.2]{K}, the pair $(L^{-1}(t), H(t))_{t \geq 0}$ is a two-dimensional subordinator, possibly killed, with Laplace exponent
\begin{equation*} 
	\mathbb{E} [e^{- q L^{-1}(1) -\theta H(1)} 1_{\{ 1 < L(\infty)\}}] = e^{- \kappa(q, \theta)},\quad q,\theta\geq 0,
\end{equation*}
where, for a positive constant $k$,
\begin{equation} \label{laplace_exp_ladder_gral}
    \log(\kappa(q, \theta)) \coloneqq \log(k)
     +\int_0^{\infty}\int_{(0,\infty)}(e^{-t}-e^{-qt-\theta x})\frac{1}{t}\mathbb{P}(Y(t)\in {\rm d} x)\, {\rm d} t.
\end{equation}

As discussed in \cite[Section 4.3]{KM95}, the continuous-time Gittins index can be expressed as $\Gamma^\infty(x) = \int_{[0,\infty)} R (x+y) \mu^\infty ({\rm d} y)$, where, thanks to \eqref{laplace_exp_ladder_gral},
\begin{align}\label{mu_laplace_classical_gral}
	\int_{[0,\infty)} e^{\textrm{i} \theta y} \mu^\infty ({\rm d} y) = \frac {\kappa(q,0)} {\kappa(q,- \textrm{i} \theta)} = \exp \left\{ \int_0^{\infty}\int_{(0,\infty)}  \frac{e^{-qs}}{s} ({e^{\mathrm{i} \theta x}}-1)\mathbb{P}\left(Y(s)\in {\rm d} x\right)  {\rm d} s \right\}, \, \theta \in \mathbb{R}.
\end{align}
\begin{proposition}\label{Lemma: GI decomp}
    For every $\lambda > 0$, the continuous-time Gittins index admits the representation:
    \begin{equation*}
        \Gamma^\infty(x) = \mathbb{E}[\Gamma^\lambda(x + \overline{Y}(e_{q + \lambda}))], \quad x \in \mathbb{R},
    \end{equation*}
    where, we write $e_{p}$ for an independent exponential random variable with rate $p > 0$, and, for $t \geq 0$, $\overline{Y}(t) \coloneqq \sup_{0 \leq s \leq t}Y(s)$.
\end{proposition}

\begin{proof}
Using Proposition \ref{imp_lem} for $r=1$ and because $T_n\sim\text{Gamma}(n,\lambda)$ with $T_n$ independent of the process $Y$, Fubini's Theorem gives
\begin{align*}
    - \log (1-\varphi(q,\theta)) &=\sum_{n=1}^{\infty}\int_0^{\infty}\int_{(0,\infty)}e^{-qs}{e^{\mathrm{i}\theta x}}\mathbb{P}\left(Y(s)\in {\rm d} x\right)\frac{\lambda(\lambda s)^{n-1}}{n\Gamma(n)}e^{-\lambda s}{\rm d} s\notag\\\
    &=\sum_{n=1}^{\infty}\int_{0}^\infty\int_{(0,\infty)}e^{-qs}{e^{\mathrm{i}\theta x}}\mathbb{P}\left(Y(s)\in {\rm d} x\right)\frac{\lambda(\lambda s)^{n-1}}{n!}e^{-\lambda s}{\rm d} s \notag\\
	&=\int_0^{\infty}\int_{(0,\infty)}e^{-qs}{e^{\mathrm{i}\theta x}}\mathbb{P}\left(Y(s)\in {\rm d} x\right)\frac{(e^{\lambda s}-1)}{s}e^{-\lambda s}{\rm d} s.
\end{align*}
Substituting this in Proposition \ref{Gitt_gral} yields
\begin{equation}\label{Eq: characteristic Ulambda}
	\widehat{\mu^{\lambda}}(\theta)=\exp \left\{\int_0^{\infty}\int_{(0,\infty)} (1- e^{-\lambda s})  \frac{e^{-qs}}{s}({e^{\mathrm{i} \theta x}}-1)\mathbb{P}\left(Y(s)\in {\rm d} x\right) {\rm d} s\right\}.
\end{equation}
Let the measure $\bar{\mu}^{\lambda}$ be specified by its Fourier transform
\begin{equation}\label{Eq: CF sup}
    \widehat{\bar{\mu}^{\lambda}}(\theta) = \frac {\widehat{\mu^{\infty}}(\theta)}  {\widehat{\mu^{\lambda}}(\theta) } = \exp\left\{\int_0^{\infty}\int_{(0,\infty)} \frac{e^{-(q + \lambda)s}}{s}({e^{\mathrm{i} \theta x}}-1)\mathbb{P}\left(Y(s)\in {\rm d} x\right) {\rm d} s\right\}.
\end{equation}
By \cite[Theorem 6.15]{K}, this matches the law of $\overline{Y}(e_{q + \lambda})$. Hence, the convolution theorem yields $\mu^\infty = \mu^\lambda * \bar{\mu}^\lambda$ and 
\begin{equation}\label{Eq: sup decom}
    \overline{Y}(e_q) \overset{d}{=} U^\lambda + V^\lambda,
\end{equation}
where $U^\lambda \sim \mu^\lambda$ and $V^\lambda \sim \overline{Y}(e_{q + \lambda})$ are independent random variables on a common probability space. Since $\Gamma^\lambda(x) = \int_{[0, \infty)} R(x + y)\, \mu^\lambda({\rm d} y) = \mathbb{E}[R(x + U^\lambda)]$, we have $\Gamma^\infty(x) = \int_{[0, \infty)} R(x + y)\, \mu^\infty({\rm d} y) = \mathbb{E}[R(x + U^\lambda + V^\lambda)] = \mathbb{E}[\Gamma^\lambda(x + \overline{Y}(e_{q + \lambda}))]$.
\end{proof}

We now impose the following assumption on the L\'evy measure $\nu$ of $Y$. When the expectation is finite, let 
\begin{equation}\label{Eq: Laplace exponent}
    \psi(\theta):=\log \mathbb{E}[e^{\theta Y(1)}], \quad \theta \geq 0.
\end{equation}
\begin{assum}\label{asm: exponential moment}
    We assume there exists some $\delta > 0$ such that $\int_{\mathbb{R}\backslash (-1, 1)} e^{\delta |x|}\, \nu({\rm d} x) < \infty$ with $\psi(\delta) < q$. 
\end{assum}
Assumption \ref{asm: exponential moment} implies the finiteness of $\mathbb{E}[e^{\delta \overline{Y}(e_{q + \lambda})}]$ for every $\lambda \geq 0$. Under Assumption \ref{asm: exponential moment}, we can bound the rate at which $\mathbb{E}[\overline{Y}(e_\lambda)]$ decays as $\lambda \uparrow \infty$. Later we show that, under (locally) Lipschitz conditions on the reward function, this rate governs the rate of convergence of the Gittins index.
\begin{lem}\label{Lemma: exp sup convergence}
    Under Assumption \ref{asm: exponential moment}, $\mathbb{E}[\overline{Y}(e_\lambda)] \xrightarrow{\lambda \uparrow \infty} 0$ at rate $\mathcal{O}(\lambda^{-1})$ if $Y$ is of bounded variation and $\mathcal{O}(\lambda^{-1/2})$ otherwise.
\end{lem}
\begin{proof}
Under Assumption \ref{asm: exponential moment}, $\mathbb{E}[e^{\delta \overline{Y}(e_{\lambda})}] < \infty$ for every $\lambda \geq q$. By differentiation with respect to $\delta$ and taking $\delta\downarrow0$, with $Y^{+}(s) \coloneqq \max(0, Y(s))$,
\begin{equation*}
    \mathbb{E}[\overline{Y}(e_{\lambda})] =  \int_0^{\infty}\int_{(0,\infty)} \frac{e^{-\lambda s}}{s} x \mathbb{P}\left(Y(s) \in {\rm d} x\right) {\rm d} s = \int_0^{\infty} e^{-\lambda s} \frac{\mathbb{E}[Y^+(s)]}{s} {\rm d} s.
\end{equation*}
Observe that $\int_1^{\infty} e^{-\lambda s} \frac{\mathbb{E}[Y^+(s)]}{s} {\rm d} s$ decays exponentially in $\lambda$. Thus, we examine the rate of convergence of $\int_0^{1} e^{-\lambda s} \frac{\mathbb{E}[Y^+(s)]}{s} {\rm d} s$.

(i). If $Y$ is of bounded variation then $\mathbb{E}[Y^+(s)] = \mathcal{O}(s)$ as $s \downarrow 0$; see \cite{Nutz11} and \cite[Lemma 4.11]{K}. Therefore,  $\int_0^{1} e^{-\lambda s} \frac{\mathbb{E}[Y^+(s)]}{s}  {\rm d} s = \mathcal{O}\left(\lambda^{-1}\right)$ as $\lambda \uparrow \infty$.

(ii). If $Y$ is of unbounded variation then $\mathbb{E}[Y^+(s)] = \mathcal{O}(s^{1/2})$ as $s \downarrow 0$; see \cite{Nutz11}. In this case, $\int_0^{1}e^{-\lambda s} s^{-1/2} \, {\rm d} s = \lambda^{-1/2} \int^{\lambda }_0 e^{-u} u^{-1/2}\, {\rm d} u$, which is $\mathcal{O}(\lambda^{-1/2})$ as $\lambda\uparrow \infty$.

\end{proof}

We consider reward functions that satisfy the following exponential growth condition.
\begin{assum}\label{asm: exponential growth}
    For some $A > 0$ and $\beta \in (0, \delta)$, where $\delta$ is as in Assumption \ref{asm: exponential moment}, $|R(x)| \leq A e^{\beta |x|}$ for $x \in \mathbb{R}$.
\end{assum}
\begin{proposition}\label{Prop: pointwise monotone convergence}
    Suppose Assumptions \ref{asm: exponential moment} and \ref{asm: exponential growth} hold. If the reward function $R$ is continuous then $\Gamma^\lambda$ converges monotonically and uniformly on compact sets to $\Gamma^\infty$.
\end{proposition}
\begin{proof}
Similar to the decomposition \eqref{Eq: sup decom}, Eq. \eqref{Eq: characteristic Ulambda} can be used to show, for any $0 < \lambda_1 < \lambda_2$, $U^{\lambda_2} \overset{d}{=} U^{\lambda_1} + Q_{\lambda_1, \lambda_2}$ where $Q_{\lambda_1, \lambda_2}$ is an independent non-negative random variable with characteristic transform given by the left-hand side of \eqref{Eq: characteristic Ulambda} upon setting $q = q + \lambda_1$ and $\lambda = \lambda_2 -\lambda_1$. Similar to the last step in the proof of Proposition \ref{Lemma: GI decomp}, $\Gamma^{\lambda_2} (x) = \mathbb{E}[\Gamma^{\lambda_1}(x + Q_{\lambda_1, \lambda_2})] \geq \Gamma^{\lambda_1}(x)$ where the inequality follows from the non-negativity of $Q_{\lambda_1, \lambda_2}$ and $\Gamma^{\lambda_1}$ being a non-decreasing function. L\'evy's continuity theorem gives $\lim_{\lambda_2 \uparrow \infty}Q_{\lambda_1, \lambda_2} = \overline{Y}(e_{q + \lambda_1})$ in distribution.
    
Under Assumption \ref{asm: exponential growth}, $\Gamma^{\lambda_1}$ grows at most exponentially with rate $\beta$. We may therefore fix some $p > 1$ such that $\beta p < \delta$, and, for some positive constant $C$, 
\begin{equation*}
    \sup_{\lambda_2 > \lambda_1} \mathbb{E}[|\Gamma^{\lambda_1}(x + Q_{\lambda_1, \lambda_2})|^p] \leq C\sup_{\lambda_2 > \lambda_1} \mathbb{E}[e^{\beta p (|x| + Q_{\lambda_1, \lambda_2})}] \leq C\sup_{\lambda_2 > \lambda_1} \mathbb{E}[e^{\beta p (|x| + U^{\lambda_2})}] \leq C \mathbb{E}[e^{\delta (|x| + \overline{Y}(e_q))}] < \infty,
\end{equation*}
In the above, the first inequality follows from the growth condition on $\Gamma^{\lambda_1}$, the second follows from $U^{\lambda_2} \overset{d}{=} U^{\lambda_1} + Q_{\lambda_1, \lambda_2}$, the third follows from \eqref{Eq: sup decom} and Assumption \ref{asm: exponential moment}. Thus, for any fixed $\lambda_1 > 0$, the family $\{\Gamma^{\lambda_1}(x + Q_{\lambda_1, \lambda_2}); \lambda_2 > \lambda_1\}$ is uniformly integrable. Since $R: \mathbb{R} \to \mathbb{R}$ is continuous, the function $\Gamma^{\lambda_1}(x) = \mathbb{E}[R(x + U^{\lambda_1})]$ is also continuous in view of dominated convergence and Assumptions \ref{asm: exponential moment} and \ref{asm: exponential growth}. Together with the continuous mapping theorem and uniform integrability, $\lim_{\lambda_2 \uparrow \infty} \Gamma^{\lambda_2} (x) = \lim_{\lambda_2 \uparrow \infty} \mathbb{E}[\Gamma^{\lambda_1}(x + Q_{\lambda_1, \lambda_2})] = \mathbb{E}[\Gamma^{\lambda_1}(x + \overline{Y}(e_{q + \lambda_1}))] = \Gamma^{\infty}(x)$. Since the convergence is pointwise and monotone, Dini's theorem implies that it is uniform on compact sets.
\end{proof}

By slightly strengthening Assumption \ref{asm: exponential growth}, we obtain the rate of convergence.
\begin{assum}\label{asm: exponential Lipschitz}
    There exists $A > 0$ and $\beta \in (0, \delta)$ such that $|R(x) - R(y)| \leq A e^{\beta (|x| \vee |y|)} |x - y|$ for $x, y \in \mathbb{R}$, where $\delta$ is as in Assumption \ref{asm: exponential moment}.
\end{assum}
\begin{proposition}\label{Prop: rate}
   Fix a compact set $E \subset \mathbb{R}$. Under Assumptions \ref{asm: exponential moment} and \ref{asm: exponential Lipschitz}, $\sup_{x\in E}|\Gamma^\infty(x) - \Gamma^\lambda(x)| \downarrow 0$ as $\lambda \uparrow \infty$, at rate $\mathcal{O}(\lambda^{-1})$ if $Y$ is of bounded variation and $\mathcal{O}(\lambda^{-1/2})$ otherwise.
\end{proposition}
\begin{proof}
We observe that $\Gamma^\lambda$ is locally Lipschitz and satisfies Assumption \ref{asm: exponential Lipschitz}. Indeed, for $x , y \in \mathbb{R}$,
\begin{equation*}
    |\Gamma^\lambda(x) - \Gamma^\lambda(y) | \leq A e^{\beta (|x| \vee |y|)}|x - y| \int_{[0, \infty)} e^{\beta z}\, \mu^\lambda({\rm d} z) \leq A' e^{\beta (|x| \vee |y|)} |x - y|, \quad A' \coloneqq A\mathbb{E}[e^{\beta U^\lambda}].
\end{equation*}
Here, $A' < \infty$ because $U^\lambda$ is stochastically dominated by $\overline{Y}(e_q)$ as in \eqref{Eq: sup decom}, and $\overline{Y}(e_q)$ has a finite exponential moment of order $\beta$ under Assumption \ref{asm: exponential moment}. This and Proposition \ref{Lemma: GI decomp} yield $0 \leq \Gamma^\infty(x) - \Gamma^\lambda(x) \leq A' e^{\beta |x|} \varepsilon(\lambda)$, where $\varepsilon(\lambda) := \mathbb{E} [e^{\beta\overline{Y}(e_{q + \lambda})} \overline{Y}(e_{q + \lambda})]$. Since $\beta < \delta$, by \eqref{Eq: CF sup},
\begin{equation}
    \varepsilon(\lambda) = \frac{{\rm d}}{{\rm d} \theta} \mathbb{E} [e^{\theta \overline{Y}(e_{q + \lambda})}] \Big|_{\theta = \beta} = h_1(\lambda) h_2 (\lambda),
\end{equation}
where $h_1(\lambda) := \mathbb{E} [e^{\beta\overline{Y}(e_{q + \lambda})}] \sim 1$ as $\lambda \uparrow \infty$ and 
\begin{equation*}
    h_2(\lambda) := \int^\infty_0 \frac{e^{-(q + \lambda) s}}{s} \sum_{i=1}^3 \mathbb{E}[f_i(Y(s))]\, {\rm d}s,
\end{equation*}
    where $f_1(x) \coloneqq x^+$, $f_2(x) \coloneqq \beta (x^+)^2$, and $f_3(x) \coloneqq x^+ (e^{\beta x^+} - 1 - \beta x^+)$. 
    
    For $i = 1$: we already have the small-time behavior of $\mathbb{E}[Y^+(s)]/s$ as $s \downarrow 0$ from the proof of Lemma \ref{Lemma: exp sup convergence}. In particular, as $\lambda \uparrow \infty$, $\int^\infty_0 e^{-(q + \lambda) s} \mathbb{E}[Y^+(s)]/s\, {\rm d}s$ decays at rate $\mathcal{O}(\lambda^{-1})$ if $Y$ is of bounded variation and $\mathcal{O}(\lambda^{-1/2})$ otherwise.
    
    For $i = 2$: by the additivity of $Y$, $\mathbb{E}[f_2(Y(s))] \leq \mathbb{E}[(Y(s))^2] = s  {\rm Var}(Y(1)) + s^2 (\mathbb{E}[Y(1)])^2 = \mathcal{O}(s)$. So $\int^\infty_0 e^{-(q + \lambda) s} \frac{\mathbb{E}[f_2(Y(s))]}{s} \, {\rm d}s = \mathcal{O}(\lambda^{-1})$ as $\lambda \uparrow \infty$.
    
    For $i = 3$: we first observe that $f_3(y) = o(y^2)$ as $y \downarrow 0$. Applying condition 1.1.1 and Theorem 1.1 of \cite{Jose08} to $f_3$, with the dominating function $g(x) = |x| e^{\beta |x|}$, yields $\lim_{s \downarrow 0} \frac{\mathbb{E}[f_3(Y(s))]}{s} = \int_{(0, \infty)} x (e^{\beta x} - 1 - \beta x)\, \nu({\rm d} x)$, where the right-hand side is finite by Assumption \ref{asm: exponential moment}. So $\int^\infty_0 e^{-(q + \lambda) s} \mathbb{E}[f_3(Y(s))]/s\, {\rm d}s = \mathcal{O}(\lambda^{-1})$.

    Hence, $h_2(\lambda) = \mathcal{O}(\lambda^{-1})$ if $Y$ is of bounded variation and $\mathcal{O}(\lambda^{-1/2})$ otherwise. Since $h_1(\lambda) \to 1$, the same rate holds for $\varepsilon(\lambda)$. The claim follows upon taking the supremum of $A'e^{\beta|x|}$ over $E$.
\end{proof}
\begin{corollary}\label{corollary: lipschitz rate}
    Suppose that Assumption \ref{asm: exponential moment} holds. If $R$ is Lipschitz continuous, then $\sup_{x \in \mathbb{R}} |\Gamma^\infty(x) - \Gamma^\lambda(x)| \downarrow 0$ as $\lambda \uparrow \infty$, at rate $\mathcal{O}(\lambda^{-1})$ if $Y$ is of bounded variation and $\mathcal{O}(\lambda^{-1/2})$ otherwise.
\end{corollary}
\subsection{General case.} 
We now consider arms driven by a general Markov process $Y$ with c\`adl\`ag paths satisfying Assumption \ref{assump_comparison}. Theorem \ref{thm_Gittins_exponential} below yields semi-explicit expressions of the Gittins index for reflected spectrally negative L\'evy and diffusion arms. With $\mathbb{T} := (T_0, T_1, \ldots)$, let 
\begin{equation}
    T_x^- := T_{\tau^-_x}  = \inf \{ T \in \mathbb{T}, T > 0: Y(T) \leq x \}. \label{def_T_x}
\end{equation}
By Assumption \ref{assump_exponential}, the counting process $M= (M(t) = \sum_{i} 1_{\{T_i \leq t\}}; t \geq 0)$  with inter-arrival times $(W_n; n \in \mathbb{N}_0)$ becomes a Poisson process independent of $Y$. The compensation formula for $M$ gives 
\begin{equation*} \begin{aligned}
    \mathbb{E}_x \Bigg[ \sum_{s=0}^{\tau_x^- - 1} e^{-q T_s} R(Y(T_s))\Bigg] - R(x) &= \mathbb{E}_x \left[ \int_{(0,T_x^-]} e^{-q s} R(Y(s)) {\rm d} M(s)\right] - \mathbb{E}_x \big[ e^{-q T_x^-} R(Y(T_x^-))\big] \\
    &= \mathbb{E}_x \Big[\int_{(0,T_x^-]} e^{-q s} R(Y(s)) 1_{\{ Y(s) > x \}} {\rm d} M(s)\Big]  \\
    &=\lambda\mathbb{E}_x \Big[ \int_0^{T_x^-} e^{-q s} R(Y(s)) 1_{\{ Y(s) > x \}} {\rm d} s\Big].
\end{aligned} \end{equation*}
To justify the second equality, note that $Y(T_s) > x$ for all $1 \leq s < \tau_x^-$ and $Y(T_x^-) \leq x$ by the definition of $T_x^-$. In view of the above computation, the following is a direct corollary of Proposition \ref{prop_opt_stopping_time}. 
\begin{thm} \label{thm_Gittins_exponential}
    Under Assumptions \ref{assump_comparison} and \ref{assump_exponential},
    \begin{align*}
    	\Gamma(x)= \frac{ R(x)+\lambda\mathbb{E}_x[\int_0^{T_x^-}e^{-qs} R(Y(s))1_{\{Y(s)>x\}}{\rm d} s ]}{ 1+\lambda\mathbb{E}_x[\int_0^{T_x^-}e^{-qs}1_{\{Y(s)>x\}}{\rm d} s ]}.
    \end{align*}
\end{thm}

\subsection{Reflected L\'evy case.} 
Another class of non-continuous processes for which Assumption \ref{assump_comparison} holds is that of reflected L\'evy processes; see \cite[Eqs. 3.24-3.26]{Kei23} for the order-preserving property. For a spectrally one-sided reflected process, the Gittins index of Theorem \ref{thm_Gittins_exponential} can be expressed in terms of the scale function. We demonstrate this by considering arms driven by reflected spectrally negative L\'evy processes with a lower boundary. The spectrally positive case, as well as cases involving an upper boundary, can be treated using the same technique. Other variants, such as refracted processes \cite{hernandez_et_al,Kyprianou_Loeffen} can be treated similarly.
\begin{exmp}[Insurance]\label{Example: insurance}
An insurance company manages $J$ policies. The surplus of policy $j$ is a spectrally negative L\'evy process $Y^j$ reflected from below at a policy-dependent level $b^j \in \mathbb{R}$. When in operation, the surplus grows continuously through premium income and jumps down at each claim. When out of operation, it does not change. The level $b^j$ is the minimum cash balance required for the policy to remain marketable under industry regulation; reflection represents the capital injection needed whenever a claim would otherwise breach $b^j$. The company runs one policy at a time and may switch only at the end of an operating cycle, whose duration is set by a policy-dependent random clock. By Remark \ref{remark_problem2}, if these durations are identically distributed across the $J$ policies, the reward may be taken as $r^j(x) = c^j x$ with $c^j > 0$, understood to be the interest accrued at surplus level $x$ under policy $j$.
\end{exmp}

For a L\'evy process $Y = (Y(t); t \geq 0)$, let $Y^{(b)} = (Y^{(b)}(t); t \geq 0)$ be $Y$ reflected from below at $b \in \mathbb{R}$:
\begin{equation*}
    Y^{(b)}(t) \coloneqq Y(t) - \inf_{0 \leq s \leq t} \left(0 \wedge (Y(s) - b)\right).
\end{equation*}
We consider arms driven by $Y^{(b)}$ and obtain the Gittins index of Theorem \ref{thm_Gittins_exponential}. For $x \geq b$, by Theorem \ref{thm_Gittins_exponential} and because $Y^{(b)}$ under $\mathbb{P}_x$ has the same law as $Y^{(b-x)}+x$ under $\mathbb{P}$,
\begin{equation} \label{Gittins_reflected}
    \Gamma(x) = \frac{R(x) + \lambda \int_{(0,\infty)} R(x + y) u^{(q)}(0, {\rm d} y; b - x)}{1 + \lambda \int_{(0,\infty)} u^{(q)}(0, {\rm d} y; b - x)},
\end{equation}
where, for $p \geq 0$, $a \leq 0$, and $z, y \in [a, \infty)$, we define
\begin{equation}
    u^{(p)}(z, {\rm d} y; a) \coloneqq \int^\infty_0 e^{-pt} \mathbb{P}_z(Y^{(a)}(t) \in {\rm d} y, t < T^-_0)\, {\rm d} t,
\end{equation}
with $T^-_0$ as in \eqref{def_T_x} for $Y^{(a)}$. We also have that $\Gamma(x) = \Gamma(b)$ for all $x < b$.

To compute $\Gamma(x)$ in \eqref{Gittins_reflected}, we require the scale functions. For $p \geq 0$, the $p$-scale function $W^{(p)}: \mathbb{R} \to [0, \infty)$ of a spectrally negative L\'evy process $Y$ is continuous and strictly increasing on the positive half-line, and satisfies
\begin{equation}
    \int_0^\infty e^{-sx} W^{(p)}(x)\, {\rm d} x = \frac{1}{\psi(s) - p}, \quad s > \Phi(p),
\end{equation}
with $\psi$ defined in \eqref{Eq: Laplace exponent} and $\Phi$ being its right-inverse. On the negative half-line, $W^{(p)}$ is set to zero. For $x \in \mathbb{R}$, define $\overline{W}^{(p)}(x) := \int_0^x W^{(p)}(y) {\rm d} y$ and $Z^{(p)}(x) \coloneqq 1 + p\overline{W}^{(p)}(x)$. Additionally, for $\lambda > 0$, 
\begin{equation*}
    Z^{(p + \lambda)}(x, \Phi(p))
    \coloneqq e^{\Phi(p) x} \left(1 + \lambda\int_0^{x} e^{-\Phi(p) z} W^{(p + \lambda)}(z)\, {\rm d} z \right).
\end{equation*}
The proof of the following result is given in Appendix \ref{Appx: Gittins reflected}.

\begin{proposition}\label{thm_gittins_reflected} 
For a spectrally negative L\'evy process reflected from below at $a\leq 0$, and for $y>0$,
\begin{equation*}
    u^{(p)}(0, {\rm d} y; a) = \frac{\Phi(p) e^{-\Phi(p)y}}{(p + \lambda)\frac{Z^{(p + \lambda)}(-a, \Phi(p))}{Z^{(p + \lambda)}(-a)} -\lambda} \, \diff y, \quad p \geq 0.
\end{equation*}
The Gittins index of a reflected spectrally negative L\'evy process with barrier $b \in \mathbb{R}$ is given by
\begin{equation}\label{Eq: gittins reflected}
    \Gamma(x)
    = \begin{cases}
        (1 - c^{(b)}_x) R(x) + c^{(b)}_x \tilde{R}(x), & x > b, \\
        \dfrac{q}{\lambda + q}R(b) + \dfrac{\lambda}{\lambda + q} \tilde{R}(b), & x \leq b,
    \end{cases}
\end{equation}
where $c^{(b)}_x := (\lambda/(q + \lambda))Z^{(q + \lambda)}(x-b)/Z^{(q + \lambda)}({x-b}, \Phi(q))$ and $\tilde{R}(z) := \int_0^\infty \Phi(q) e^{-\Phi(q)y} R(z + y)\, {\rm d} y$, which is the continuous-time Gittins index of the spectrally negative L\'evy process (see \cite[Eq. (31)]{PYM}).
\end{proposition}

Denote by $\Gamma_b^\lambda$ the Gittins index of $Y^{(b)}$ to emphasize the dependence on $b$ and $\lambda$, the following proposition gives the asymptotic behavior of $\Gamma_b^\lambda$ as $\lambda \uparrow \infty$. In particular, observe that the result holds for functions $R$ that may not be bounded or continuous.
\begin{proposition}\label{Prop: reflected convergence}
    For any $R: \mathbb{R} \to \mathbb{R}$ such that the integrals in \eqref{Gittins_reflected} are finite on $[b, \infty)$, $\lim_{\lambda \uparrow \infty}\Gamma_b^\lambda(x) = \Gamma^\infty(x) := \tilde{R}(x)$ in Proposition \ref{thm_gittins_reflected}. 
    Also, $|\Gamma_b^\lambda(x) - \Gamma^\infty(x)|$ is $\mathcal{O}(\Phi(q + \lambda)^{-1})$ if $x > b$ and $\mathcal{O}(\lambda^{-1})$ if $x = b$.
\end{proposition}
\begin{proof}
    In view of $\tilde{R}$ in Proposition \ref{thm_gittins_reflected}, the rate at which $|\Gamma_b^\lambda(x) - \Gamma^\infty(x)| \to 0$ as $\lambda \uparrow \infty$ is the rate that
\begin{equation}\label{Eq: reflected convergence rate}
    \bigg|1 - \frac{\lambda}{q + \lambda} \frac{Z^{(q + \lambda)}(x-b)}{Z^{(q + \lambda)}({x-b}, \Phi(q))}\bigg| \to 0, \quad \lambda \uparrow \infty.
\end{equation}
For $x = b$, $Z^{(q + \lambda)}(x-b) = Z^{(q + \lambda)}({x-b}, \Phi(q)) = 1$. So \eqref{Eq: reflected convergence rate} holds with rate $\mathcal{O}(\lambda^{-1})$ as $\lambda \uparrow \infty$. Consider $x > b$. By \cite[Cor. 8.9]{K}, for any $p > 0$,
\begin{equation}\label{Eq: W scale exponential}
    W^{(p)}(x) = \frac{\exp(\Phi(p)x)}{\psi'(\Phi(p))} - \hat{u}^{(p)}(x), \quad x \geq 0,
\end{equation}
where $\hat{u}^{(p)}$ is the density of the $p$-potential measure of the dual process $\hat{Y} \coloneqq -Y$.
By \eqref{Eq: W scale exponential} and the definition of $Z^{(q + \lambda)}$,
\begin{align*}
    Z^{(q + \lambda)}(x) &= 1 + (q + \lambda) \int^{x}_0 \left(\frac{\exp(\Phi(q + \lambda)y)}{\psi'(\Phi(q + \lambda))} - \hat{u}^{(q + \lambda)}(y) \right)\, \diff y\\
    &\sim \frac{q + \lambda}{\psi'(\Phi(q + \lambda))} \frac{e^{\Phi(q + \lambda)x}}{\Phi(q + \lambda)}, \quad \lambda\uparrow \infty. 
\end{align*}
Similarly, by \eqref{Eq: W scale exponential} and the definition of $Z^{(q + \lambda)}(\cdot, \Phi(q))$, 
\begin{align*}
    Z^{(q + \lambda)}(x, \Phi(q)) &= e^{\Phi(q) x} \left(1 + \lambda\int_0^{x} e^{-\Phi(q) y} \left(\frac{\exp(\Phi(q + \lambda)y)}{\psi'(\Phi(q + \lambda))} - \hat{u}^{(q + \lambda)}(y)\right)\, {\rm d} y \right)\\
    &\sim \frac{\lambda}{\psi'(\Phi(q + \lambda))}\frac{\exp(\Phi(q + \lambda)x)}{\Phi(q + \lambda) - \Phi(q)}, \quad \lambda\uparrow \infty.
\end{align*}
Hence, 
\begin{equation*}
    \frac{\lambda}{q + \lambda}\frac{Z^{(q + \lambda)}(x-b)}{Z^{(q + \lambda)}({x-b}, \Phi(q))} \sim \frac{\Phi(q + \lambda) - \Phi(q)}{\Phi(q + \lambda)}, \quad \lambda \uparrow \infty.
\end{equation*}
We see that \eqref{Eq: reflected convergence rate} holds with rate $1/\Phi(q + \lambda)$ as $\lambda \uparrow \infty$.
\end{proof}

\subsection{Diffusion case.}
Suppose that $Y$ is a one-dimensional diffusion process taking values in an interval $I$ with end points $-\infty\leq l_1 < l_2\leq \infty$, defined as the solution to the SDE:
\begin{align}\label{sde}
    {\rm d} Y(t)=b(Y(t))\, {\rm d} t+\sigma(Y(t)) \, {\rm d} B(t),\quad t>0,
\end{align}
with initial condition $Y(0) = x$ under $\mathbb{P}_x$, where $B = (B(t); t\geq0)$ is the standard Brownian motion. We assume that, for $x,y\in I$,
\begin{align*}
    |b(x)-b(y)|+|\sigma(x)-\sigma(y)|&\leq K|x-y|, \quad b^2(x)+\sigma^2(x) \leq K^2(1+x^2),
\end{align*}
for a positive constant $K$ which assures that there exists a unique strong solution to \eqref{sde}. Using a comparison theorem for diffusion processes (see \cite[Corollary 3.1]{PZ}), we see that the process $Y$ satisfies Assumption \ref{assump_comparison}. Thus, Theorem \ref{thm_Gittins_exponential} can be applied to compute the Gittins index for $Y$.

\begin{exmp}[Competing technologies]\label{Example: tech}
A decision maker allocates resources among $J$ competing technologies. Following \cite{Lennart25}, let $-Y^{j}$ be the log unit production cost of technology $j$, evolving only while resources are allocated to the technology and frozen otherwise. For a mature technology with stable cost, $-Y^{j}$ is well described by a mean-reverting process such as the Ornstein-Uhlenbeck process (see \cite[Eq. (18)]{Lennart25}), and $Y^j$ is then again Ornstein-Uhlenbeck. The reward may be modeled by the strictly increasing $R^{j}(x) = -e^{-x}$. Resource commitments typically last for a technology-dependent period during which supply chains and production capacity cannot be redirected.
\end{exmp}

To compute the Gittins index, we introduce two of the basic characteristics of the diffusion process $Y$: the speed measure $m$ and the scale function $s$. They are given by 
\begin{equation*}
    m({\rm d} x):=\frac{2e^{B(x)}}{\sigma^2(x)} {\rm d} x,\quad s(x):=\int^xe^{-B(y)} {\rm d} y, 
\end{equation*}
for $l_1< x< l_2$, where $B(x):=\int^x\frac{2b(y)}{\sigma^2 (y)} {\rm d} y$. For $\alpha>0$, we denote by $\psi_{\alpha}$ (resp. $\varphi_{\alpha}$) the increasing (resp. decreasing) positive solution to the ODE
\begin{equation}\label{ODE}
    \frac{1}{2}\sigma^2(x) u''(x) + b(x) u'(x)= \alpha u(x).
\end{equation}
The functions $\psi_{\alpha}$ and $\varphi_{\alpha}$ are the fundamental solutions to \eqref{ODE}. Denote by $w_{\alpha}$ the Wronskian, which is independent of $x$ and can be written as
\begin{align*}
w_{\alpha}:=\psi_{\alpha}^+(x)\varphi_{\alpha}(x)-\psi_{\alpha}(x)\varphi_{\alpha}^+(x)
=\psi_{\alpha}^-(x)\varphi_{\alpha}(x)-\psi_{\alpha}(x)\varphi_{\alpha}^-(x),
\end{align*}
where $\psi_{\alpha}^+$ (resp. $\psi_{\alpha}^-$) is the right (resp. left) derivative of $\psi_{\alpha}$ with respect to the scale function $s$, i.e. 
\begin{align*}
    \psi_{\alpha}^+(x)&\coloneqq \lim_{h\downarrow 0}\frac{\psi_{\alpha}(x+h)-\psi_{\alpha}(x)}{s(x+h)-s(x)}, \quad \psi_{\alpha}^-(x) \coloneqq  \lim_{h\downarrow 0}\frac{\psi_{\alpha}(x-h)-\psi_{\alpha}(x)}{s(x-h)-s(x)}.
\end{align*}
The function $\varphi_{\alpha}^+$ (resp. $\varphi_{\alpha}^-$) is defined similarly.
The Green function $G_{\alpha}$ of $Y$ with respect to the speed measure $m$, can be written as 
\begin{align*}
    G_{\alpha}(x,y) &=
    \begin{cases}
    	w_{\alpha}^{-1}\psi_{\alpha}(x) \varphi_{\alpha}(y),\qquad x\leq y, \\
        w_{\alpha}^{-1}\psi_{\alpha}(y)\varphi_{\alpha}(x),\qquad y\leq x.
    \end{cases}
\end{align*}
The proof of Theorem \ref{thm_gittins_diffusion} is presented in Appendix \ref{Appx: Gittins diffusion}.
\begin{thm} \label{thm_gittins_diffusion}
	The Gittins index in Theorem \ref{thm_Gittins_exponential} becomes
	\begin{align}
		\Gamma(x)= \frac{H^{(q,\lambda)}(x; R)}{H^{(q,\lambda)}(x; 1)},  \quad x \in I,
	\end{align}
	where, for $f(z)=R(z)$ and $f(z) = 1$,
\begin{align*}
	H^{(q,\lambda)}(x; f)
    &\coloneqq f(x)\left(1-\lambda \int_{I \cap (x, \infty)}\frac{\varphi_{q}(z)}{\varphi_q(x)}G_{q+\lambda}(x,z) \, m({\rm d} z)\right) + \lambda\int_{I\cap(x,\infty)}f(z)G_{q+\lambda}(x,z) \, m({\rm d} z)\notag\\
    &+ \lambda^2\int_{I\cap (x, \infty)} G_{q+\lambda}(x,z) \int_{I \cap (x, \infty)} f(u)   \left(G_{q}(z,u)-\frac{\varphi_q(z)}{\varphi_q(x)} G_q(x,u)\right) \, m({\rm d} u) \, m({\rm d} z). \notag
\end{align*}
\end{thm}
In practice, $\Gamma(x)$ is evaluated by numerical quadrature over a truncated domain. The truncation error is governed by the decay of $R(z)\varphi_q(z)m({\rm d}z)$.

\section{Numerical experiments}\label{Sect: experiment}
This section reports numerical experiments illustrating the optimality result of Section \ref{sec_exponential}. Each arm evolves according to one of five models: a Brownian motion (BM), a reflected BM (RBM), an Ornstein-Uhlenbeck (OU) process, a spectrally negative L\'evy process with exponential jumps (SNLP), or a reflected spectrally negative L\'evy process with exponential jumps (RSNLP). In each experiment, all arms follow the same model, share the sigmoid reward function $R(x) = (1 + e^{-x})^{-1}$, and share the exponential holding time $G\sim \text{Exp}(\lambda)$. The arms only differ in their parameters, which are reported in Tables \ref{Table: notations} and \ref{Table: parameters}. The analytical expressions for the Gittins indices of the five models are given in Appendix \ref{Appx: Gittins explicit}. \footnote{The Python code used to produce the numerical results of this section is available at: \newline https://github.com/qingyuanzhang3/RandomInterventionMAB .}

Using Monte Carlo simulation, we compare the random-intervention Gittins index (GI) strategy against one of two benchmarks: the continuous-intervention Gittins index strategy and the myopic strategy. The former replaces the random-intervention index by its continuous-intervention counterpart (see Section \ref{sec_exponential}). The latter selects the arm with the highest current reward; since the arms share a common increasing reward function, the arm with the largest current state.

Figure \ref{fig: GI comparison} reports the performance gap between the GI strategy and the selected benchmark, together with $t$-statistics for testing its significance, in heatmaps over different combinations of discount rates $q$, numbers of arms $J$, and intervention rates $\lambda$. For each pair of $(q, \lambda)$, we truncate the horizon at $T_{q, \lambda}$. Since the sigmoid reward is bounded by $1$, the truncation error is bounded by 
\begin{equation*}
    e^{-q T_{q, \lambda}} \mathbb{E}\left[\sum^\infty_{n = 0} e^{-q T_n}\right] = e^{-q T_{q, \lambda}} \sum^\infty_{n = 0} \left(\frac{\lambda}{\lambda + q}\right)^n.
\end{equation*}
The value of $T_{q, \lambda}$ is chosen so that this bound is of order $10^{-7}$.

Figure \ref{fig: GI comparison} is consistent with our theoretical results. The
GI strategy outperforms both benchmarks at every $(q, J, \lambda)$, and the gaps remain
significant at the $95\%$ level after a Bonferroni correction for multiple comparisons. The figure also illustrates the sensitivity of the gap with respect to $q$, $J$, and $\lambda$. 

\begin{table}[h]
\centering\small
\caption{Model notations and descriptions.}
\label{Table: notations}
\begin{tabular}{@{}p{0.09\textwidth} p{0.83\textwidth}@{}}
\toprule
\textbf{Model} & \textbf{Description} \\
\midrule
\addlinespace[2pt]
BM & Brownian motion with diffusion coefficient $\sigma>0$, denoted $\text{BM}(\sigma)$. \\
\addlinespace[2pt]
RBM & $\text{BM}(\sigma)$ reflected from below at $b < 0$, denoted $\text{RBM}(b, \sigma)$. \\
\addlinespace[2pt]
OU & Ornstein-Uhlenbeck process solving \eqref{sde} with $b(y)=-\gamma y$ and $\sigma(y)=1$, denoted $\text{OU}(\gamma)$. \\
\addlinespace[2pt]
SNLP & Spectrally negative L\'evy process with exponential jumps, $Y(t) = \mu t + \sigma B(t)-\sum_{n=1}^{N(t)} Z_n$, where $B$ is a standard Brownian motion, $N$ a Poisson process of rate $\ell$, and $\{Z_n\}_{n \geq 1} \overset{\text{i.i.d.}}{\sim} \text{Exp}(r)$; denoted $\text{SNLP}(\mu,\sigma,\ell,r)$. \\
\addlinespace[2pt]
RSNLP & $\text{SNLP}(\mu,\sigma,\ell,r)$ reflected at the barrier $b\in\mathbb{R}$, denoted $\text{RSNLP}(b,\mu,\sigma,\ell,r)$. \\
\midrule
\addlinespace[2pt]
\multicolumn{2}{@{}p{0.92\textwidth}@{}}{For every model $x \in \mathbb{R}$ denotes the initial value of the arm.} \\
\bottomrule
\end{tabular}
\vspace{2pt}
\end{table}

\begin{table}[h]
\centering\small
\caption{Simulation and model parameters}
\label{Table: parameters}
\begin{tabular}{@{}p{0.09\textwidth} p{0.20\textwidth} p{0.63\textwidth}@{}}
\toprule
\textbf{Model} & \textbf{Fixed parameters} & \textbf{Arm parameters (arms $1,\ldots,8$)} \\
\midrule
\addlinespace[2pt]
BM
  & ---
  & $x_i$: $-1.98$, $-3.89$, $-2.12$, $-3.33$, $-2.28$, $-2.97$, $-2.47$, $-2.69$ \newline
  $\sigma_i$: $0.72$, $10$, $1.23$, $6.73$, $1.81$, $4.78$, $2.53$, $3.47$ \\
\addlinespace[2pt]
RBM & $b = -5$ 
    & $x_i$: $-0.42$, $-2.56$, $-0.66$, $-2.32$, $-0.9$, $-2.08$, $-1.13$, $-1.85$ \newline
    $\sigma_i$: $1$, $4$, $1.33$, $3.67$, $1.67$, $3.33$, $2$, $3$ \\
\addlinespace[2pt]
OU & $\sigma = 1$
    & $x_i$: $2$, $0.5$, $1$, $1.5$, $1$, $-0.5$, $-1$, $-1.5$ \newline
    $\gamma_i$: $0.5$, $4$, $4$, $4$, $0.5$, $4$, $4$, $4$ \\
\addlinespace[2pt]
SNLP & $\mu = 0.1$, $\ell = r = 1$
    & $x_i$: $-2.45$, $-3.89$, $-2.58$, $-3.53$, $-2.71$, $-3.26$, $-2.87$, $-3.05$ \newline
    $\sigma_i$: $1.4$, $10$, $2.07$, $7.56$, $2.79$, $5.87$, $3.61$, $4.6$ \\
\addlinespace[2pt]
RSNLP
  & $b = -10$, $\mu = 0.1$, $\ell = r = 1$
  & $x_i$: $-0.22$, $-0.79$, $-0.28$, $-0.72$, $-0.34$, $-0.66$, $-0.4$, $-0.59$ \newline
    $\sigma_i$: $1$, $2$, $1.11$, $1.89$, $1.22$, $1.78$, $1.33$, $1.67$ \\
\midrule
\addlinespace[2pt]
All
  & $\lambda \in \{0.5, 1, 1.5\}$
  & $(q,T)$: $(0.1,180)$, $(0.2,90)$, $(0.3,60)$, $(0.4,50)$, $(0.5,40)$, $(0.6,30)$, $(0.7,30)$, $(0.8,20)$.\\
\bottomrule
\end{tabular}
\vspace{2pt}
\end{table}

\begin{figure*}[t!]
  \centering
  \includegraphics[width=\textwidth, height=\textheight, keepaspectratio]{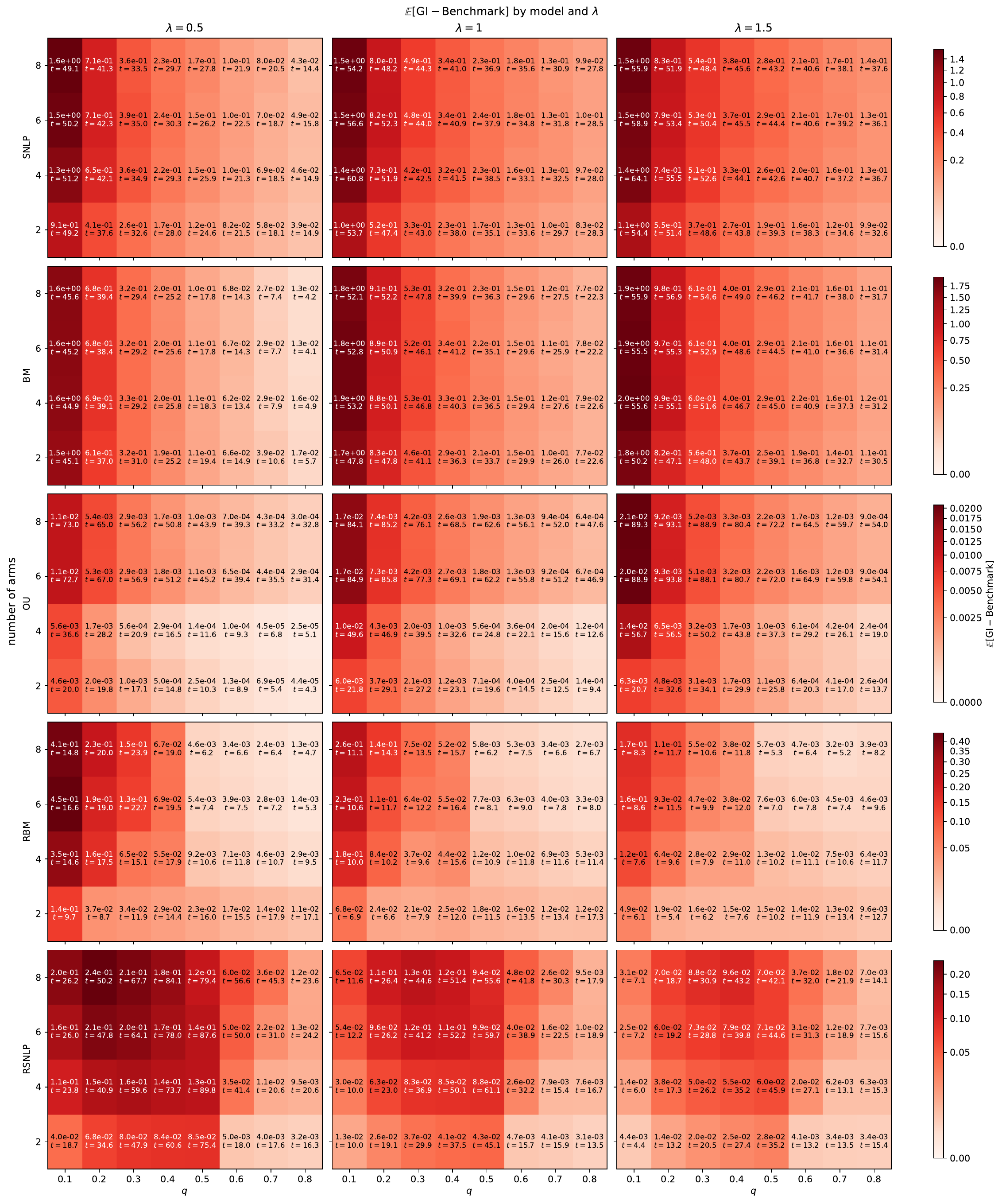}
  \caption{$\mathbb{E}[\mathrm{GI}-\mathrm{Benchmark}]$ by model with $t$-statistics. Benchmarks: the myopic strategy for SNLP, BM, and OU; the continuous-intervention Gittins index for RBM and RSNLP.}
  \label{fig: GI comparison}
\end{figure*}

\section{Conclusion}\label{Sect: conclusion}
This work studies the multi-armed bandit problem in which decisions are made at random-discrete times. We establish a tractable characterization of the Gittins index for general L\'evy arms. We then relate the index to its continuous-intervention counterpart, establishing monotone convergence and its rate as interventions become frequent. We also derive explicit expressions of the Gittins index for reflected spectrally negative L\'evy and diffusion processes.

A key assumption is the monotonicity of the reward function (Assumption \ref{assump_comparison}). It is standard in the existing literature and arises naturally in Examples \ref{Example: wireless}, \ref{Example: insurance}, and \ref{Example: tech}. However, for some control applications, non-monotone rewards are natural. For example, in target tracking, the reward attains its maximum at the desired operating point $x^* \in \mathbb{R}$, i.e., $R(x)=-(x-x^{*})^{2}$. 

It is a natural research direction to relax this assumption. Theorem \ref{Thm: sec2 Gittin index general} identifies the Gittins index strategy as optimal, and monotonicity is used only to determine the optimal stopping time in \eqref{Gamma_x_Markov}, which is then a down-crossing time. Relaxing the assumption therefore reduces to determining the geometry of the continuation region from the shape of the reward together with the arm dynamics. For a unimodal reward and suitable arm dynamics, such as a L\'evy process, one would expect an exit time from an interval or a ray. Identifying the boundary values then requires additional structure, such as symmetry of both the reward and the arm. A systematic treatment of such reward functions, and of the classes of Markov processes for which the associated stopping time can be identified, remains open.

Another potential direction for extending the current model is to allow state-dependent holding times. Instead of holding an arm for i.i.d. random times from a fixed distribution, the holding times $(W_s^j; s \in \mathbb{N}_0)$ may each be the first jump time of a Cox process with state-dependent intensity $\omega^j(Y^j(\cdot))$, i.e., the time at which the integrated intensity first exceeds an independent threshold $e_s^j$, where $\omega^j$ is  measurable and $(e_s^j) \overset{\text{i.i.d.}}{\sim} \text{Exp}(1)$. The challenge with this extension is that $(W_s^j; s \in \mathbb{N}_0)$ may not be i.i.d., and new fluctuation identities need to be derived to calculate the Gittins index.

Finally, we note that certain control applications are better captured by the restless bandit framework, where unselected arms do not remain frozen. For example, in scheduling and queueing, unserved classes continue to receive arrivals. In maintenance planning, unattended machines deteriorate between inspections. Combining the restless model with random intervention times is a natural continuation of this work. The relevant index is then the Whittle index, and we expect the fluctuation-theoretic approach developed here to remain effective.

\section*{acknowledgements}
K. Noba was supported by JSPS KAKENHI grant JP21K13807 and JSPS Open Partnership Joint Research Projects grant JPJSBP120209921. In addition, Kei Noba stayed at Centro de Investigaci\'on en
Matem\'aticas in Mexico as a JSPS Overseas Research Fellow and received support for the research environment there. J. L. P\'erez gratefully acknowledges the support of the Fulbright
Program during his sabbatical stay at Arizona State University. He also wishes
to thank the faculty and staff of the School of Mathematical and Statistical Sciences at
ASU for their hospitality. K. Yamazaki was supported by JSPS KAKENHI grants JP20K03758, JP24K06844, and JP24H00328.

\bibliographystyle{abbrv}        
\bibliography{citations2}           

\appendix
\section{Proof of results in section 3}
\subsection{Proof of Lemma \ref{WF}.} \label{proof_wf}
Fix $0 < r < 1$ and $p = 1-r$. As in the classical Wiener-Hopf factorization for random walks (see, for instance, \cite{greenwood}), using the lack of memory of the geometric random variable $\Xi_{p}$ and the strong Markov property, 
	\begin{align}\label{desc_1}
		(T_G,S_G) \overset{(d)}{=}
		\sum_{i=1}^{\nu}(\tau^i,H^i),
	\end{align}
	where $\nu$ is an $\mathbb{N}_0$-valued geometric random variable with parameter $p_+:=\mathbb{P}(N> \Xi_{p})$, with $(T_G,S_G)=(0,0)$ for the case $\nu=0$. Additionally, $\{(\tau^i, H^i)\}_{i=1}^{\infty}$ is a sequence of i.i.d.\ random variables, whose common distribution equals that of $(T_N,S_N)$ conditioned on the event 
    \[A:= \{N\leq \Xi_{p}\}.
    \]
In addition, $\nu$ is independent of $\{(\tau^i, H^i)\}_{i=1}^{\infty}$. Then, 
	\begin{align}\label{desc_0}
    \begin{split}
		\mathbb{E}(e^{-qT_N+\mathrm{i} \theta S_N};A)
        &=\sum_{n=1}^{\infty}\mathbb{E}(e^{-qT_n+\mathrm{i} \theta S_n}; N=n)r^n  =\varphi(q,\theta; r).
    \end{split}
	\end{align}
In particular, setting $q = \theta = 0$ gives
	\begin{align}\label{prob_q+}
		q_+ := 1 - p_+ = \mathbb{P}(A)
		=\varphi(0,0; r).
	\end{align}
	Hence, by \eqref{desc_0} and \eqref{prob_q+},
	\begin{align*}
	\mathbb{E}\Big[e^{-q \tau^i +\mathrm{i} \theta H^i } \Big] = \mathbb{E} \left[ e^{-qT_N+\mathrm{i} \theta S_N}| A \right] = {\varphi(q,\theta;r)} / {q_+}.
	\end{align*}
	Using \eqref{desc_1}, \eqref{desc_0}, \eqref{prob_q+}, and the fact that $\nu$ is an independent geometric random variable with parameter $p_+$, 
	\begin{align*}
		\mathbb{E}\left[e^{-qT_G+\mathrm{i} \theta S_G}\right]
        = \mathbb{E}\left[ \mathbb{E}\left[e^{\sum_{i=1}^{\nu} (-q \tau^i+\mathrm{i} \theta H^i)} | \nu \right]\right]
        =\mathbb{E}\left[\left(\frac{\varphi(q,\theta;r)}{\varphi(0,0;r)}\right)^{\nu}\right]=\frac{p_+}{1-\varphi(q,\theta;r)}.
	\end{align*}
On the other hand, applying a similar argument for $(T_D,S_D)$, we obtain
	\begin{align*}
		\mathbb{E}\left[e^{-qT_D+\mathrm{i} \theta S_{D}}\right]=\mathbb{E}\left[\left(\frac{\hat{\varphi}(q,\theta;r)}{\hat{\varphi}(0,0;r)}\right)^{\hat{\nu}}\right]=\frac{p_-}{1-\hat{\varphi}(q,\theta;r)},
	\end{align*}
where  $p_- :=\mathbb{P}(\hat{N}> \Xi_{p})$ and $\hat{\nu}$ is a geometric random variable with parameter $p_-$.
	Finally, recalling \eqref{kappa_varphi},
	\begin{align*}
		\mathbb{E}\left[e^{-qT_{\Xi_{p}}+\mathrm{i} \theta S_{\Xi_{p}}}\right] =\sum_{k=0}^{\infty}\mathbb{E}\left[e^{-qT_k+\mathrm{i} \theta S_k}\right]r^{k}p
		=\sum_{k=0}^{\infty}\mathbb{E}\left[e^{-qT_1+\mathrm{i} \theta S_1}\right]^kr^{k}p=\frac{p}{1-r\kappa(q,\theta)}.
	\end{align*}
Now let us consider the random pair $(T_{\Xi_{p}}-T_G,S_{\Xi_{p}}- S_G)$. Using Feller's Duality Lemma for random walks (cf. \cite{Feller}) and the fact that $T_k$ is the sum of $k$ independent random variables we obtain that $\{(T_n-T_{n-k},S_{n-k}-S_n):k=0,1,\dots,n\}$ has the same distribution as $\{(T_k,-S_k):k=0,1,\dots,n\}$. Because the first visit to the maximum $G$ for the original random walk corresponds to the last visit to the minimum of the reversed random walk which has the same law as $-S$, $(T_{\Xi_{p}}-T_G,S_G-S_{\Xi_{p}})\overset{\mathcal{L}}{=}(T_D,-S_D)$ or equivalently $(T_{\Xi_{p}}-T_G,S_{\Xi_{p}}-S_G)\overset{\mathcal{L}}{=}(T_D,S_D)$. Since $\Xi_{p}$ is independent of the process $(T_n,S_n)_{n\geq1}$, and the latter has independent and stationary increments, the random pairs $(T_G, S_G)$ and $(T_{\Xi_{p}}-T_G,\, S_{\Xi_{p}}-S_G)$ are independent. Thus,
	\begin{align*}
\mathbb{E}\left[e^{-qT_{\Xi_{p}}+\mathrm{i} \theta S_{\Xi_{p}}}\right]
        &= \mathbb{E}\left[e^{-q (T_{\Xi_{p}}- T_G) +\mathrm{i} \theta (S_{\Xi_{p}}-S_G)}\right] \mathbb{E}\left[e^{-q T_G +\mathrm{i} \theta S_G}\right] = \mathbb{E}\left[e^{-q T_D +\mathrm{i} \theta S_D}\right]\mathbb{E}\left[e^{-q T_G +\mathrm{i} \theta S_G}\right]  
    \end{align*}
or
	\begin{align}\label{spitzer_ran}
		\frac{p}{1-r\kappa(q, \theta)}=\frac{p_-}{1-\varphi(q,\theta;r)}\frac{p_+}{1-\hat{\varphi}(q,\theta;r)}.
	\end{align}
	By Spitzer's identity for random walks (see identity (1) in \cite{greenwood}), we have that
	\begin{align}\label{spitzer}
		p=1-r=(1-\mathbb{E}\left[r^N\right])(1-\mathbb{E}\big[r^{\hat{N}}\big])
        =(1-q_+)(1-q_-)=p_+p_-.
	\end{align}
	Hence using \eqref{spitzer_ran} together with \eqref{spitzer} gives \eqref{spitzer_ran_1} for $0 < r < 1$.  The case $r=1$ holds by  dominated convergence.

\subsection{Proof of Proposition \ref{imp_lem}.}\label{use_lemma}

Let $E_- := (-\infty, 0]\times[0,\infty)$, $E_+ :=  (0, \infty) \times[0,\infty)$, and $E:= E_- \cup E_+ = \mathbb{R} \times [0, \infty)$. Define, for any $n \in \mathbb{N}_0$ and any Borel set $I\times J\in\mathcal{B} (E)$, the following measure
\begin{align}
    \nu_n(I\times J) :=\mathbb{P}(N \geq n, S_n\in I \cap (-\infty,0], T_n\in J)
    =\mathbb{P}(N > n, S_n\in I, T_n\in J).
\end{align}
In particular, $\nu_0 (I\times J):=1_{\{(0, 0)\in I\times J\}}$, and $\nu_n(E_+) = 0$. Additionally, for the joint distribution of $(N,T_{N},S_N)$ we denote 
\begin{align*}
	\tilde{\nu}_n(I\times J):=\mathbb{P}(N=n, S_N\in I, T_N\in J).
\end{align*}
By the definition of $N$, $\tilde{\nu}_n(E_-) = 0$. Because $(T_n, S_n; n  \in \mathbb{N}_0)$ has independent and stationary increments, 
\begin{align*}
    \nu_{n+1}(I\times J) &= \int_{E_-}\nu_n({\rm d} y,{\rm d}  s)\nu((I-y)\times(J-s)), \qquad \textrm{for } I\times J\in\mathcal{B}(E_-),\\
    \tilde{\nu}_{n+1}(I\times J) &= \int_{E_-}\nu_n({\rm d} y,{\rm d} s)\nu((I-y)\times(J-s)),\qquad \textrm{for } I\times J\in\mathcal{B}(E_+),
\end{align*}
where $\nu(I\times J):=\mathbb{P}(S_1 \in I, T_1\in J)$, $I\times J \in 
	\mathcal{B} (E)$.
In particular, for the case $n=0$, we have that $\nu_{1}(I\times J)=\nu(I\times J)$ for $I\times J\in\mathcal{B}(E_-)$ and $\tilde{\nu}_1(I\times J)=\nu(I\times J)$ for $I\times J\in\mathcal{B}(E_+)$. Summing these and because $\tilde{\nu}_{n+1}(E_-) = \nu_{n+1}(E_+)=0$, for $I\times J\in\mathcal{B}(E)$,
\begin{multline}\label{lem_1}
	\tilde{\nu}_{n+1}(I\times J)+\nu_{n+1}(I\times J)\\
    =\int_{E_-}\nu_n({\rm d} y,{\rm d} s)\nu((I-y)\times(J-s))
    =\int_{E}\nu_n({\rm d} y,{\rm d} s)\nu((I-y)\times(J-s)).
\end{multline}
 Now, for $r \in (0,1)$, $q \geq 0$ and $\theta\in\mathbb{R}$, we have
\begin{align}
\begin{split}
	\varphi(q,\theta; r)&=\sum_{n=1}^{\infty}r^n
	\int_{E_+}e^{-qt+\mathrm{i}\theta y}\tilde{\nu}_n({\rm d} y,{\rm d} t)=\sum_{n=1}^{\infty}r^n
	\int_{E}e^{-qt+\mathrm{i}\theta y}\tilde{\nu}_n({\rm d} y, {\rm d} t), \label{xi_expression}
\end{split}\\ 
\begin{split}
    \gamma(q,\theta; r)&:= \sum_{n=0}^{\infty}r^n 
	\int_{E_-}e^{-qt+\mathrm{i}\theta y}\nu_n({\rm d} y,{\rm d} t)= \sum_{n=0}^{\infty}r^n
	\int_{E}e^{-qt+\mathrm{i}\theta y}\nu_n({\rm d} y, {\rm d} t). \label{gamma_expression}
\end{split}
\end{align}
Recalling $\nu_0 (I\times J):=1_{\{(0, 0)\in I\times J\}}$ and $\kappa(q,\theta) = \int_{E}e^{-qs+\mathrm{i} \theta x}\nu({\rm d} s,{\rm d} x)$, we have by \eqref{lem_1} and the convolution theorem,
\begin{align*}
    \varphi(q,\theta; r) + \gamma(q,\theta; r)&= 1+\sum_{n=1}^{\infty}r^n \int_{E}e^{-qt+\mathrm{i} \theta y} \int_{E}\nu_{n-1}({\rm d} y',{\rm d} s')\nu(({\rm d} y-y')\times({\rm d} t-s')) \\
    &= 1+\sum_{n=1}^{\infty}r^n\int_{E}e^{-qt+\mathrm{i}\theta y} \nu_{n-1}({\rm d} y, {\rm d} t) \int_{E}e^{-qt+ \mathrm{i}\theta y} \nu({\rm d} y,{\rm d} t) \\
    &= 1+ r \kappa(q,\theta) \sum_{n=0}^{\infty}r^n \int_{E}e^{-qt+\mathrm{i}\theta y} \nu_{n}({\rm d} y,{\rm d} t) \\
    &= 1+ r \gamma(q,\theta; r) \kappa(q,\theta).
\end{align*}
Rearranging the terms  and taking logarithms gives
\begin{align}\label{lem_3}
	-\log (1-r\kappa(q,\theta))= - \log (1-\varphi(q,\theta; r)) 
    + \log \gamma(q,\theta; r).
\end{align}
For $|x|<1$,  Taylor expansion gives
\begin{align*}
-\log(1-x)=\sum_{n=1}^{\infty}\frac{x^{n}}{n}, \quad \log(1+x)=\sum_{n=1}^{\infty}(-1)^{n+1}\frac{x^n}{n}. 
\end{align*}
Applying this and the convolution theorem, the left hand side of \eqref{lem_3} becomes
\begin{multline}\label{lem_7}
	-\log (1-r\kappa(q,\theta)) = \sum_{n=1}^{\infty}\frac{r^n}{n} \kappa^n (q,\theta)
    =\sum_{n=1}^{\infty}\frac{r^n}{n}\int_{E}e^{-qt+\mathrm{i}\theta x}\nu^{*n}({\rm d} x,{\rm d} t) =K(E_+) + K(E_-)
\end{multline}
where $\nu^{*n}$ is the $n$-th convolution power of $\nu$ and 
\[
K(A) := \sum_{n=1}^{\infty}\frac{r^n}{n}\int_{A}e^{-qt+\mathrm{i}\theta x}\mathbb{P}\left(S_n\in {\rm d} x,T_n\in {\rm d} t\right).
\]
Applying Taylor exapansion on the right hand side of \eqref{lem_3}, and then \eqref{lem_7} gives
\begin{align}\label{lem_5}
    K(E_+) + K(E_-) =\sum_{n=1}^{\infty}\frac{\varphi^n(q,\theta; r)}{n}
    +\sum_{n=1}^{\infty} \frac{(-1)^{n+1}}{n}(\gamma(q,\theta; r)-1)^n.
\end{align}
Note that $K(E_+)$ and $K(E_-)$ are the Fourier transforms of finite measures on $E_+$ and $E_-$, respectively. Applying a similar argument to \eqref{xi_expression} and \eqref{gamma_expression}, we have that $\varphi(q, \theta; r)$ and $\gamma(q, \theta; r)$ are the Fourier transforms of two finite measures on $E_+$ and $E_-$, respectively. By the convolution theorem, $\varphi^n(q, \theta; r)$ is the Fourier transform of a measure on $E_+$, and therefore so is the first series on the right-hand side of \eqref{lem_5}. Again using the convolution theorem, $(\gamma(q,\theta;r)-1)^n$ is the Fourier transform of a measure on $E_-$. Since the second series on the right-hand side of \eqref{lem_5} is absolutely convergent, we may separate it into its positive and negative parts. Each part is the Fourier transform of a finite measure on $E_-$, and therefore the original series is the Fourier transform of a measure on $E_-$. By the uniqueness of the Fourier transform, and by restricting to the set $E_+$, we have
\begin{align*}
	K(E_+) = \sum_{n=1}^{\infty}\frac{ \varphi^n(q,\theta; r)}{n} = - \log (1-\varphi(q,\theta; r)),
\end{align*}
which completes the proof for $r < 1$. The case of $r=1$ holds by monotone convergence.

\section{Proof of results in section 4}

\subsection{Proof of Proposition \ref{thm_gittins_reflected}.}\label{Appx: Gittins reflected}
We first review some fluctuation identities which are required for the proof.

For $x, y \in \mathbb{R}$, define
\begin{align*}
    \mathscr{Z}^{(p,\lambda)}(x, y) \coloneqq Z^{(p + \lambda)}(x - y)
    - \lambda\int^x_0W^{(p)}(x - z) Z^{(p + \lambda)}(z - y)\, {\rm d} z,
\end{align*}
which, by identity (6) in \cite{Loeffen}, can also be written as 
\begin{align}\label{Eq: W, Z scr}
    \mathscr{Z}^{(p,\lambda)}(x, y) = Z^{(p)}(x - y)
    + \lambda\int_{y}^{0} W^{(p)}(x - u)Z^{(p + \lambda)}(u - y)\, {\rm d} u.
\end{align}
Fix $a \leq 0$. Define the first passage times of the L\'evy process $Y$ for $c \in \mathbb{R}$: 
\begin{align} \label{first_passage_time}
\begin{split}
    \kappa_c^- &\coloneqq \inf\{t >  0: Y(t) < c\}, \quad \kappa_c^+ \coloneqq \inf \{ t > 0: Y(t) > c\}.
\end{split}
\end{align} 
By Lemma 2.1 in \cite{Loeffen}, for $0 < x < c$,
\begin{align*}
    \mathbb{E}_x [e^{-p\kappa_0^-} Z^{(p + \lambda)}(Y(\kappa_0^-) - a); \kappa_0^- < \kappa_c^+]
    = \mathscr{Z}^{(p,\lambda)}(x, a) - \frac{W^{(p)}(x)}{W^{(p)}(c)} \mathscr{Z}^{(p,\lambda)}(c, a).
\end{align*}
The following limits are known, for $x \in \mathbb{R}$,
\begin{equation}\label{Eq: W, Z limits}
    \lim_{y\to\infty} \frac{W^{(p)}(y + x)}{W^{(p)}(y)} = e^{\Phi(p)x}, \quad \lim_{y\to\infty} \frac{W^{(p)}(y)}{Z^{(p)}(y)} = \frac{\Phi(p)}{p}. 
\end{equation}
Using \eqref{Eq: W, Z limits} together with \eqref{Eq: W, Z scr}, and taking $c \uparrow \infty$, 
\begin{align*}
   \mathbb{E}_x [e^{-p\kappa_0^-} Z^{(p + \lambda)}(Y(\kappa_0^-) - a)] &= \lim_{c \to\infty} \mathbb{E}_x [e^{-p\kappa_0^-} Z^{(p + \lambda)}(Y(\kappa_0^-) - a); \kappa_0^- < \kappa_c^+]\\
   &= \mathscr{Z}^{(p,\lambda)}(x, a) - W^{(p)}(x) \left(\frac{p}{\Phi(p)}e^{-\Phi(p)a} + \lambda\int^0_a e^{-\Phi(p)z}Z^{(p + \lambda)}(z - a)\, {\rm d} z\right).
\end{align*}
Using integration by parts,
\begin{align*}
    e^{-\Phi(p)a} + \frac{\lambda}{p} \Phi(p)\int^0_a e^{-\Phi(p)z}Z^{(p + \lambda)}(z - a)\, {\rm d} z    = \frac{p + \lambda}{p} Z^{(p + \lambda)}(-a, \Phi(p)) - \frac{\lambda}{p} Z^{(p + \lambda)}(-a).
\end{align*}
Thus, 
\begin{align}\label{Eq: loffen Z}
\begin{split}
    B(x,a)&:=\frac {\mathbb{E}_x [e^{-p\kappa_0^-} Z^{(p + \lambda)}(Y(\kappa_0^-) - a)]} {Z^{(p + \lambda)}(-a)} = \frac {\mathscr{Z}^{(p,\lambda)}(x, a) } {Z^{(p + \lambda)}(-a)}- W^{(p)}(x)\left(\frac{p + \lambda}{\Phi(p)} \frac{Z^{(p + \lambda)}(-a, \Phi(p))} {Z^{(p + \lambda)}(-a)} - \frac{\lambda}{\Phi(p)}\right).
\end{split}
\end{align}
We also use the following to obtain a concise expression.
\begin{lem} \label{lemma_laplace} 
We have
\begin{equation*}
    \int^\infty_0 e^{-\Phi(p + \lambda)z} \mathscr{Z}^{(p,\lambda)}(z, a)\, {\rm d} z
    =\frac{p+\lambda}{\lambda \Phi(p+\lambda)}e^{-\Phi(p+\lambda)a}.
\end{equation*}
\end{lem}
\begin{proof}
Substituting \eqref{Eq: W, Z scr}, we have $\int^\infty_0 e^{-\Phi(p + \lambda)z} \mathscr{Z}^{(p,\lambda)}(z, a)\, {\rm d} z= D_1(a) + D_2(a)$, where
\begin{align}
   D_1(a) &:= \int^\infty_0 e^{-\Phi(p + \lambda)z} Z^{(p)}(z - a)\, {\rm d} z, \notag\\
   D_2(a) &:= \int^\infty_0 e^{-\Phi(p + \lambda)z} \lambda  \int_{a}^{0} W^{(p)}(z - u)Z^{(p + \lambda)}(u - a)\, {\rm d} u\, {\rm d} z. \label{Eq: laplace transform Z scr}
\end{align}
By integration by parts,
\begin{equation}\label{Eq: integration by parts 0}
    D_1(a)
    = \frac{p}{\lambda\Phi(p + \lambda)} Z^{(p)}(-a, \Phi(p + \lambda)) + \frac{Z^{(p)}(-a)}{\Phi(p + \lambda)},
\end{equation}
where
\begin{align*}
    Z^{(p)}(x, \Phi(p + \lambda)) &\coloneqq e^{\Phi(p + \lambda)x} \left(1 - \lambda\int_0^{x} e^{-\Phi(p + \lambda)z} W^{(p)}(z)\, {\rm d} z \right) = \lambda\int_0^{\infty} e^{-\Phi(p + \lambda)z} W^{(p)}(z + x) \,{\rm d} z.
\end{align*}
Now, for $D_2$ as defined in \eqref{Eq: laplace transform Z scr}, by integration by parts,
\begin{align}\label{Eq: integration by parts 1}
\begin{split}
    \int_{a}^{0} W^{(p)}(z - u)Z^{(p + \lambda)}(u - a)\, {\rm d} u &=-\overline{W}^{(p)}(z)Z^{(p + \lambda)}(- a)+\overline{W}^{(p)}(z-a)\\
    &+(p+\lambda)\int_a^0\overline{W}^{(p)}(z-u)W^{(p + \lambda)}(u - a)\, {\rm d} u.
\end{split}
\end{align}	
Also note that
\begin{align}
\begin{split}
    \lambda \int^\infty_0e^{-\Phi(p+\lambda)z}\overline{W}^{(p)}(z){\rm d} z &= \frac{1}{\Phi(p+\lambda)},
\end{split}\label{Eq: integration by parts 2}\\
\begin{split}
    \lambda\int^\infty_0e^{-\Phi(p+\lambda)z}\overline{W}^{(p)}(z-a){\rm d} z&=\frac{\lambda}{\Phi(p+\lambda)}\overline{W}^{(p)}(-a)+\frac{Z^{(p)}(- a,\Phi(p+\lambda))}{\Phi(p+\lambda)}\\
    &=\frac{\lambda}{\Phi(p+\lambda)}\left(\frac{Z^{(p)}(-a)-1}{p}\right)+\frac{Z^{(p)}(- a,\Phi(p+\lambda))}{\Phi(p+\lambda)}.
    \end{split}\label{Eq: integration by parts 3}
\end{align}
By \eqref{Eq: integration by parts 3} with $a=u$, 
\begin{align}
\begin{split}\label{Eq: D2 misc 1}
    &\lambda(p + \lambda) \int^\infty_0e^{-\Phi(p+\lambda)z}\int_a^0\overline{W}^{(p)}(z-u)W^{(p + \lambda)}(u - a)\, {\rm d} u\, {\rm d} z\\
    &=(p+\lambda) \int_a^0W^{(p + \lambda)}(u - a) \left(\frac{\lambda}{\Phi(p+\lambda)}\left(\frac{Z^{(p)}(-u)-1}{p}\right)+\frac{Z^{(p)}(- u,\Phi(p+\lambda))}{\Phi(p+\lambda)}\right)\, {\rm d} u\\
    &=\frac{\lambda(p+\lambda)}{p\Phi(p+\lambda)} \left(\int_a^0W^{(p + \lambda)}(u - a)Z^{(p)}(-u){\rm d} u \right.\\
    &\left.-\overline{W}^{(p+\lambda)}(-a) + \frac{p}{\lambda} \int_a^0W^{(p + \lambda)}(u - a)Z^{(p)}(- u,\Phi(p+\lambda)){\rm d} u\right)\\
    &= \frac{1}{\Phi(p+\lambda)} \left(\frac{p+\lambda}{p}\left(Z^{(p+\lambda)}(-a)-Z^{(p)}(-a)\right) \right.\\
    &\left.-\frac{\lambda}{p}\left(Z^{(p+\lambda)}(-a)-1\right)\right) +\frac{p+\lambda}{\lambda\Phi(p+\lambda)} \left(Z^{(p+\lambda)}(- a,\Phi(p+\lambda))-Z^{(p)}(- a,\Phi(p+\lambda))\right),
\end{split}
\end{align}
where in the last equality we used identity (A.3) in \cite{PPY}. Substituting \eqref{Eq: integration by parts 1} into \eqref{Eq: laplace transform Z scr}, then substituting \eqref{Eq: integration by parts 2}, \eqref{Eq: integration by parts 3}, and \eqref{Eq: D2 misc 1}, we obtain 
\begin{align*}
  D_2(a) = - \frac{1}{\Phi(p+\lambda)} \left(Z^{(p)}(-a) -  \frac{p+\lambda}{\lambda}Z^{(p+\lambda)}(- a,\Phi(p+\lambda)) + \frac{p}{\lambda} Z^{(p)}(- a,\Phi(p+\lambda))\right).
\end{align*}
The above identity, together with \eqref{Eq: integration by parts 0}, yields the desired expression.
\end{proof}
Now for the proof of Proposition \ref{thm_gittins_reflected},
we want to compute, for $x \in [a, \infty)$, $y > 0$,
\begin{align}
    u^{(p)}(x, {\rm d} y; a) \coloneqq \int^\infty_0 e^{-pt} \mathbb{P}_x(Y^{(a)}(t) \in {\rm d} y, t < T^-_0) \, {\rm d} t
    = p^{-1} \mathbb{P}_x(Y^{(a)}(\mathbf{e}_{p}) \in {\rm d} y, \mathbf{e}_{p} < T^-_0),
\end{align}
where $\mathbf{e}_{p}$ is an independent exponential random variable with parameter $p$.

(i) For $y > 0$: as $a < 0$, we have $Y^{(a)}(t) = Y(t)$ on $[0, \kappa_0^-)$. Hence, $\inf\{t \geq 0: Y^{(a)}(t) < 0\} = \inf\{t \geq 0: Y(t) < 0\} =: \kappa_0^-$;  see \eqref{first_passage_time}. Let
\begin{align} \label{mu_def}
\begin{split}
    \mu_x^p ({\rm d} z) &:= \mathbb{P}_x \left(Y^{(a)}(\kappa_0^-)\in {\rm d} z, \kappa_0^- < \mathbf{e}_{p}\right) = \begin{cases}
        \mathbb{P}_x \left(Y(\kappa_0^-)\in {\rm d} z, \kappa_0^- < \mathbf{e}_{p}\right), & z > a, \\
        \mathbb{P}_x \left(Y(\kappa_0^-)\in (-\infty, a], \kappa_0^- < \mathbf{e}_{p}\right), & z = a.
    \end{cases}
\end{split}
\end{align}
Let $\mathbf{e}_{\lambda} = T_1 \sim \exp(\lambda)$, which is independent of $Y$ and $\mathbf{e}_{p}$. By conditioning on whether $\kappa_0^-$ occurs before $\mathbf{e}_{p}$, 
\begin{align*}
    u^{(p)}(x, {\rm d} y; a) = \frac{1}{p} \mathbb{P}_x(Y^{(a)}(\mathbf{e}_{p}) \in {\rm d} y; \mathbf{e}_{p} < \kappa_0^-)
    + \frac{1}{p}\int_{[a, 0]} 
    \mu_x^p ({\rm d} z)
    \mathbb{P}_z(Y^{(a)}(\mathbf{e}_{p}) \in {\rm d} y, \mathbf{e}_{p} < T_0^-).
\end{align*}
By Theorem 2.8(i), (iii) in \cite{KKR}, for $z \in [a, 0]$, 
\begin{align}\label{eqn_tau_0_e_q_lambda}
    \mathbb{P}_z(\kappa_0^+ < \mathbf{e}_{p} \wedge \mathbf{e}_{\lambda}) = \mathbb{E}_z\left[e^{-(p + \lambda) \kappa_0^+}; \kappa_0^+ < \infty\right]
    = \frac{Z^{(p + \lambda)}(z - a)}{Z^{(p + \lambda)}(-a)}, 
\end{align}
and
\begin{align}
    \begin{split}
        \frac{1}{p}\mathbb{P}_z(Y^{(a)}(\mathbf{e}_{p}) \in {\rm d} y, \mathbf{e}_{p} < \kappa_0^+ \wedge \mathbf{e}_{\lambda})
        &= \mathbb{E}_z \left[ \int^\infty_0 e^{-pt} 1_{\{Y^{(a)}(t) \in {\rm d} y, t < \kappa^+_0\}} 1_{\{t < \mathbf{e}_{\lambda}\}}\, {\rm d} t \right] \\ 
        &= \mathbb{E}_z \left[ \int^\infty_0 e^{-(p+\lambda)t} 1_{\{Y^{(a)}(t) \in {\rm d} y, t < \kappa^+_0\}}  {\rm d} t \right] \\
        &= \left(\frac{Z^{(p + \lambda)}(z - a)}{Z^{(p + \lambda)}(- a)} W^{(p + \lambda)}(-y) - W^{(p + \lambda)}(z - y)\right)\, {\rm d} y.
        \end{split} \label{eqn_tau_0_e_q_lambda_P} 
\end{align}
Using \eqref{eqn_tau_0_e_q_lambda},
\begin{align*}
   \frac{1}{p}\mathbb{P}_z(Y^{(a)}(\mathbf{e}_{p}) \in {\rm d} y, \mathbf{e}_{p} < T_0^-)
    &= \mathbb{P}_z(\kappa_0^+ < \mathbf{e}_{p} \wedge \mathbf{e}_{\lambda}) u^{(p)}(0, {\rm d} y; a) = \frac{Z^{(p + \lambda)}(z - a)}{Z^{(p + \lambda)}(-a)}u^{(p)}(0, {\rm d} y; a),
\end{align*}
and, by Theorem 2.7(iii) in \cite{KKR},
\begin{align*}
    \frac{1}{p} \mathbb{P}_x(Y^{(a)}(\mathbf{e}_{p}) \in {\rm d} y; \mathbf{e}_{p} < \kappa_0^-)
    &= \frac{1}{p} \mathbb{P}_x(Y(\mathbf{e}_{p}) \in {\rm d} y; \mathbf{e}_{p} < \kappa_0^-) = \left(e^{-\Phi(p)y} W^{(p)}(x) - W^{(p)}(x - y)\right) {\rm d} y.
\end{align*}
Hence, using \eqref{mu_def},
\begin{align}
\begin{split}
    u^{(p)}(x, {\rm d} y; a)
    &= \int_{[a, 0]} 
    \mu_x^p ({\rm d} z)
    \frac{Z^{(p + \lambda)}(z - a)}{Z^{(p + \lambda)}(-a)}u^{(p)}(0, {\rm d} y; a) + \left(e^{-\Phi(p)y} W^{(p)}(x) - W^{(p)}(x - y)\right) {\rm d} y\\
    &= B(x,a) u^{(p)}(0, {\rm d} y; a)+ \left(e^{-\Phi(p)y} W^{(p)}(x) - W^{(p)}(x - y)\right) {\rm d} y. 
    \end{split}\label{Eq: u(x, dy), y > 0}
\end{align}
(ii) Now, we compute $u^{(p)}(0, {\rm d} y; a)$ for $y > 0$. By conditioning on whether $\mathbf{e}_{p}$ is smaller than $\mathbf{e}_{\lambda}$, 
\begin{align*}
    u^{(p)}(0, {\rm d} y; a) = \frac{1}{p} \mathbb{P}(Y^{(a)}(\mathbf{e}_{p}) \in {\rm d} y, \mathbf{e}_{p} < \mathbf{e}_{\lambda})
    + \int^\infty_0 \mathbb{P}(Y^{(a)}(\mathbf{e}_{\lambda}) \in {\rm d} z, \mathbf{e}_{\lambda} < \mathbf{e}_{p}) u^{(p)}(z, {\rm d} y; a).
\end{align*}
Similar to \eqref{eqn_tau_0_e_q_lambda_P},
\begin{align*}
    \frac{1}{p} \mathbb{P}_z(Y^{(a)}(\mathbf{e}_{p}) \in {\rm d} y, \mathbf{e}_{p} < \kappa_c^+ \wedge \mathbf{e}_{\lambda})
    = \left(\frac{Z^{(p + \lambda)}(z - a)}{Z^{(p + \lambda)}(c - a)} W^{(p + \lambda)}(c -y) - W^{(p + \lambda)}(z - y)\right)\, {\rm d} y.
\end{align*}
Taking $c \uparrow \infty$, we have, by \eqref{Eq: W, Z limits},
\begin{align*}
    \frac{1}{p}\mathbb{P}_z(Y^{(a)}(\mathbf{e}_{p}) \in {\rm d} y, \mathbf{e}_{p} < \mathbf{e}_{\lambda}) 
    &= \lim_{c \to \infty} \frac{1}{p} \mathbb{P}_z(Y^{(a)}(\mathbf{e}_{p}) \in {\rm d} y, \mathbf{e}_{p} < \kappa_c^+ \wedge \mathbf{e}_{\lambda})\\
    &= \left(\frac{\Phi(p + \lambda)}{p + \lambda}e^{-\Phi(p + \lambda)(y - a)} Z^{(p + \lambda)}(z - a)  - W^{(p + \lambda)}(z - y)\right)\, {\rm d} y.
\end{align*}
In addition, $\mathbb{P}(Y^{(a)}(\mathbf{e}_{\lambda}) \in {\rm d} z, \mathbf{e}_{\lambda} < \mathbf{e}_{p}) = \frac{\lambda}{p}\mathbb{P}(Y^{(a)}(\mathbf{e}_{p}) \in {\rm d} z, \mathbf{e}_{p} < \mathbf{e}_{\lambda})$ by the same argument as \eqref{eqn_tau_0_e_q_lambda_P}. Hence, for $y > 0$,
\begin{align*}
    u^{(p)}(0, {\rm d} y; a) &= \frac{\Phi(p + \lambda)}{p + \lambda}e^{-\Phi(p + \lambda)(y - a)} Z^{(p + \lambda)}(- a) \, {\rm d} y\\
    &+ \lambda\int^\infty_0 \frac{\Phi(p + \lambda)}{p + \lambda}e^{-\Phi(p + \lambda)(z - a)}  Z^{(p + \lambda)}(- a) u^{(p)}(z, {\rm d} y; a) \, {\rm d} z.
\end{align*}
Substituting in \eqref{Eq: u(x, dy), y > 0} with $x = z$ and 
\begin{align*}
    \int^\infty_0  e^{-\Phi(p + \lambda)z} \left(e^{-\Phi(p)y} W^{(p)}(z) - W^{(p)}(z - y)\right) \, {\rm d} z   = \lambda^{-1} \left( e^{-\Phi(p)y} - e^{-\Phi(p+\lambda )y} \right),
\end{align*}
 we obtain
\begin{align*}
    A u^{(p)}(0, {\rm d} y; a) &= \frac{\Phi(p + \lambda)}{p + \lambda}e^{-\Phi(p + \lambda)(y - a)} Z^{(p + \lambda)}(- a) \, {\rm d} y\\
    &+ \lambda \int^\infty_0  \left(\frac{\Phi(p + \lambda)}{p + \lambda}e^{-\Phi(p + \lambda)(z - a)} Z^{(p + \lambda)}(- a)\right)  \left(e^{-\Phi(p)y} W^{(p)}(z) - W^{(p)}(z - y)\right) \, {\rm d} z\, {\rm d} y \\
    &=  \frac{\Phi(p + \lambda)}{p + \lambda}e^{\Phi(p + \lambda)a - \Phi(p)y} Z^{(p + \lambda)}(- a) \, {\rm d} y,
\end{align*}
where we recall \eqref{Eq: loffen Z} for the expressions of $B(z,a)$ and 
\begin{align*}
    A\coloneqq 1 - \lambda \int^\infty_0
    \left(\frac{\Phi(p + \lambda)}{p + \lambda}e^{-\Phi(p + \lambda)(z - a)} Z^{(p + \lambda)}(- a)\right) B(z,a) \, {\rm d} z.
\end{align*}
Substituting \eqref{Eq: loffen Z}, the Laplace transform of $W^{(p)}$ and Lemma \ref{lemma_laplace}, it simplifies to
\begin{align*}
    A = \frac{\Phi(p + \lambda)}{p + \lambda} e^{\Phi(p+\lambda)a} Z^{(p + \lambda)}(-a) \left(\frac{p + \lambda}{\Phi(p)} \frac{Z^{(p + \lambda)}(-a, \Phi(p))}{Z^{(p + \lambda)}(-a)} - \frac{\lambda}{\Phi(p)}\right). 
\end{align*}
The proof is complete after substituting this expression for $A$ to solve for $u^{(p)}(0, {\rm d} y; a)$.

\subsection{Proof of Theorem \ref{thm_gittins_diffusion}.}\label{Appx: Gittins diffusion} 
Let $\kappa_x :=\inf\{t\geq0: Y(t)=x\}$. As in page 18 of \cite{Borodin}, we have
\begin{equation}\label{exit_diff}
    \mathbb{E}_x \left[e^{-\alpha\kappa_y} \right] =  \frac{\varphi_\alpha(x)}{\varphi_\alpha(y)}, \quad x \geq y.
\end{equation}
The potential measure of the process $Y$ is given in terms of the Green function. For $\alpha \geq 0$, $x\in I$, as in page 19 of \cite{Borodin},
\begin{align}\label{res_diff_1}
    \mathbb{E}_x\left[\int_0^{\infty} e^{-\alpha t}1_{\{Y(t)\in {\rm d} z\}} \, {\rm d} t\right]=G_{\alpha}(x,z) \, m({\rm d} z).
\end{align}
By Lemma 1 in \cite{CLZ}, the potential measure of the process $Y$ killed at the first hitting time of $y$ is given by, for $x \geq y$,
\begin{align}\label{res_diff_2}
    \mathbb{E}_x\left[\int_0^{\kappa_y} e^{-\alpha t}1_{\{Y(t)\in {\rm d} z\}}\, {\rm d} t\right]
    = \left(G_{\alpha}(x, z)-\frac{\varphi_{\alpha}(x)}{\varphi_{\alpha}(y)}G_{\alpha}(y,z)\right) \, m({\rm d} z).
\end{align}
Theorem \ref{thm_gittins_diffusion} is a direct consequence of the following and Theorem \ref{thm_Gittins_exponential}.

\begin{proposition}\label{res_per_diff}
For $p > 0$, $y\in I$, and a measurable function $h: I \to \mathbb{R}$ such that the following left-hand side expectation is finite,
\begin{align*}
    &\mathbb{E}_y\left[\int_0^{T_y^-}e^{-ps}h(Y(s))\, {\rm d} s\right] \\
    &= \Theta(y)^{-1}\int_{I}G_{p + \lambda}(y,z)h(z) \, m(\diff z)\\
    &+ \Theta(y)^{-1} \int_{I \cap (y, \infty)}G_{p + \lambda}(y,z) \Bigg(\lambda\int_{I \cap (y, \infty)}h(u) \left(G_p(z,u)-\frac{\varphi_p(z)}{\varphi_p(y)}G_p(y,u)\right) m({\rm d} u)\Bigg)m({\rm d} z),
\end{align*}
where $\Theta(y) := 1-\lambda \int_{I \cap (y, \infty)}\frac{\varphi_{p}(z)}{\varphi_p(y)}G_{p + \lambda}(y,z) \,m({\rm d} z)$.
\end{proposition}
\begin{proof}
Let $g(x,y) \coloneqq \mathbb{E}_x [\int_0^{T_y^-} e^{-ps}h(Y(s)) \, {\rm d} s ]$ for $x,y\in\mathbb{R}$. For $x > y$, by the strong Markov property, \eqref{exit_diff} and \eqref{res_diff_2},
\begin{align}\label{res_diff_x>y}
\begin{split}
    g(x,y)
    &=\mathbb{E}_x\left[\int_0^{\kappa_y}e^{-ps}h(Y(s)){\rm d} s\right]+ g(y,y)\mathbb{E}_x\left[e^{-p\kappa_y};\kappa_y<\infty\right]\\
    &=\int_{I\cap (y, \infty)}h(z)\left(G_p(x,z)-\frac{\varphi_p(x)}{\varphi_p(y)}G_p(y,z)\right)m({\rm d} z)+g(y,y)\frac{\varphi_p(x)}{\varphi_p(y)}.
\end{split}
\end{align}
We now compute $g(y,y)$. To this end, note that the strong Markov property at the time $\mathbf{e}_{\lambda}$, together with \eqref{res_diff_1}, \eqref{res_diff_2}, and \eqref{exit_diff} give
\begin{align}\label{res_diff_3}
\begin{split}
  g(y,y)
    &=\mathbb{E}_{y}\left[\int_0^{\mathbf{e}_{\lambda}}e^{-ps}h(Y(s)){\rm d} s\right]+\mathbb{E}_{y}\left[e^{-p\mathbf{e}_{\lambda}}g(Y(\mathbf{e}_{\lambda}),y)1_{\{Y(\mathbf{e}_{\lambda})>y\}}\right]\\
    &=\mathbb{E}_y\left[\int_0^{\infty}e^{-(p+\lambda)s}h(Y(s))\, {\rm d} s\right]+\lambda\mathbb{E}_y\left[\int_0^{\infty}e^{-(p+\lambda)s}g(Y(s),y)1_{\{Y(s)>y\}} \, {\rm d} s\right]\\
    &=\int_{I}h(z)G_{p + \lambda}(y,z)\, m({\rm d} z)+\lambda\int_{I}g(z,y)G_{p + \lambda}(y,z)1_{\{z>y\}}\, m({\rm d} z).
\end{split}
\end{align}
Applying \eqref{res_diff_x>y} in \eqref{res_diff_3} gives
\begin{align*}
    g(y,y)
    &=\int_{I}h(z)G_{p + \lambda}(y,z) \, m({\rm d} z) + \lambda g(y,y)\int_{I}\frac{\varphi_p(z)}{\varphi_p(y)}G_{p + \lambda}(y,z) 1_{\{z>y\}} \, m({\rm d} z)\notag\\
    &+\lambda\int_{I} G_{p + \lambda}(y,z) 1_{\{z>y\}}\int_{I \cap (y, \infty)} 
    h(u)\left(G_p(z,u) - \frac{\varphi_p(z)}{\varphi_p(y)}G_p(y,u)\right) \, m({\rm d} u) \, m({\rm d} z).
\end{align*}
Solving for $g(y,y)$, we obtain the desired expression.  
\end{proof}

\section{Additional details for numerical experiments}
\subsection{Gittins index.}\label{Appx: Gittins explicit}
We provide the expressions needed to evaluate the Gittins index for the five
models of Section \ref{Sect: experiment}.
For a diffusion arm, Theorem \ref{thm_gittins_diffusion} requires the speed measure $m({\rm d} y)$ on $I$, the Green function $G_\alpha(x, y)$ for $x, y \in I$, and the ratio $\varphi_p(x)/\varphi_p(y)$ for $x, y \in I$ with $x \geq y$. For a (reflected) spectrally negative L\'evy arm, Section 5.2 of \cite{PYM} and Proposition \ref{thm_gittins_reflected} require only the scale function $W^{(p)}$ for $p > 0$. We note that the Brownian motion is both a diffusion and a spectrally negative L\'evy process. Hence, we provide both characterizations of its Gittins index. Throughout, $\alpha, p > 0$.

\smallskip
\noindent Ornstein-Uhlenbeck. The following expressions can be found in \cite[Appendix 1]{Borodin}. For $\gamma > 0$ and $x, y \in \mathbb{R}$, 
    \begin{align*}
        &m({\rm d} y) = 2e^{-\gamma y^2}\,{\rm d} y, \quad \frac{\varphi_p(x)}{\varphi_p(y)} = e^{\frac{\gamma}{2}(x^2 - y^2)} \frac{D_{-p/\gamma}(x\sqrt{2\gamma})}{D_{-p/\gamma}(y\sqrt{2\gamma})},\\
        &G_\alpha(x, y) = \frac{\Gamma(\alpha/\gamma)}{2\sqrt{\gamma \pi}} e^{\frac{\gamma}{2}(x^2 + y^2)} D_{-\alpha/\gamma}((x \vee y)\sqrt{2\gamma})  D_{-\alpha/\gamma}(-(x \wedge y)\sqrt{2\gamma}),
    \end{align*}
    where, for $\nu > 0$, $D_{-\nu}$ is the parabolic cylinder function
    \begin{equation*}
        D_{-\nu}(z) \coloneqq \frac{1}{\Gamma(\nu)} e^{-\frac{z^2}{4}} \int^\infty_0 t^{\nu - 1} e^{-zt - \frac{t^2}{2}}\, {\rm d} t, \quad z\in \mathbb{R}. 
    \end{equation*}

\smallskip
\noindent Brownian motion. For $\sigma > 0$ and $x, y \in \mathbb{R}$, as a diffusion process, the following expressions can be found in \cite[Appendix 1]{Borodin},
\begin{equation*}
    m({\rm d} y) = \frac{2}{\sigma^2} \,{\rm d} y, \quad \frac{\varphi_p(x)}{\varphi_p(y)} = \exp\left(-\frac{\sqrt{2p}}{\sigma}(x - y)\right), \quad G_\alpha(x, y) = \frac{\sigma}{2\sqrt{2\alpha}} \exp\left(-\frac{\sqrt{2\alpha}}{\sigma}|x - y|\right).
\end{equation*}
As a spectrally negative L\'evy process, the Laplace exponent \eqref{Eq: Laplace exponent} is $\psi(s) = \sigma^2 s^2/2$, its right-inverse is $\Phi(p) = \sqrt{2p}/\sigma$, and
\begin{equation*}
    W^{(p)}(x) = 1_{\{x \geq 0\}} \frac{\Phi(p)}{p} \sinh(\Phi(p)x).
\end{equation*}

\smallskip
\noindent Spectrally negative L\'evy models with exponential jumps. By Example 3.1 of \cite{Egami}, 
\begin{equation*}
    W^{(p)}(x) = 1_{\{x \geq 0\}} \left(\frac{e^{\Phi(p) x}}{\psi'(\Phi(p))} - \sum^{2}_{i = 1} B_{i, p} e^{-\xi_{i, p} x} \right), 
\end{equation*}
where $B_{i, p} = -1/\psi'(-\xi_{i, p})$ and $-\xi_{1, p}, -\xi_{2, p}$ are the
two negative roots of $\psi(s) = p$.

\smallskip
\noindent Reflected models. The Gittins indices of the reflected Brownian motion and the reflected spectrally negative L\'evy process with exponential jumps follow from Proposition \ref{thm_gittins_reflected}, on substituting the corresponding scale function above.

\end{document}